\documentclass[11pt]{amsart}

\usepackage{graphicx}
\usepackage{framed}

\usepackage[margin=1.4in]{geometry}

\usepackage[utf8]{inputenc}	

\usepackage[all]{xy}
\usepackage{pb-diagram, pb-xy}

\usepackage{amssymb}
\usepackage{amsthm}

\usepackage{tikz}
\usepackage{pgffor}

\usepackage{hyperref}

\usepackage{calc}

\usepackage[framemethod=tikz]{mdframed}

\usepackage{enumitem}
\setlist[itemize]{noitemsep}

\newtheoremstyle{saetze} 
    {5pt}                    
    {5pt}                    
    {\itshape}                   
    {12pt}                           
    {\bfseries}                   
    {.}                          
    {.5em}                       
    {}  

\theoremstyle{saetze}
\newtheorem{thm}{Theorem}[section]
\newtheorem{lem}[thm]{Lemma}
\newtheorem{cor}[thm]{Corollary}
\newtheorem{prop}[thm]{Proposition}

\theoremstyle{definition}
\newtheorem{definition}[thm]{Definition}
\newtheorem{example}[thm]{Example}
\newtheorem{remark}[thm]{Remark}

\newcommand{\g}{{\mathfrak{g}}}

\newcommand{\frakg}{{\mathfrak g}}

\newcommand{\calF}{{\mathcal F}}

\newcommand{\calN}{{\mathcal N}}

\newcommand{\calT}{{\mathcal T}}

\newcommand{\Rep}{{\mathrm{Rep}}}

\newcommand{\End}{{\mathit{End}}}

\newcommand{\one}{{\mathbf{1}}}

\newcommand{\Z}{{\mathbf{Z}}}

\def\N{\mathbb{N}}

\DeclareMathOperator{\sdim}{sdim}

\newcommand{\Tmnss}{\mathcal{T}_{m|n}^+/\mathcal{N}}
\newcommand{\Tmn}{\mathcal{T}_{m|n}}

\newcommand{\T}{\mathcal{T}}

\DeclareMathSizes{12}{11}{8}{5} 
\usepackage{xypic,dsfont}

\makeatletter
\patchcmd{\l@section}{1.0em}{\z@}{}{}
\makeatother

\hypersetup{
    colorlinks,
    citecolor=red,
    filecolor=black,
    linkcolor=blue,
    urlcolor=black
}

\long\def\comment#1{}

\DeclareRobustCommand{\gobblefive}[5]{}

\makeatletter
\def\part{\@startsection{part}{1}%
\z@{.7\linespacing\@plus\linespacing}{.5\linespacing}%
{\large\scshape\centering}}
\makeatother

\begin{document}

\title[Classical tensor categories II]{Semisimplifications and representations of the General Linear Supergroup}

\subjclass[2020]{17B10, 17B20, 17B55, 18D10, 20G05.}

\keywords{Representations of supergroups, Tensor categories, Super tannakian categories, Semisimplification, Fusion rules, Tensor product decomposition}

\author{{\rm Th. Heidersdorf, R. Weissauer}}

\address{T. H.: School of Mathematics, Statistics and Physics, Newcastle University}
\email{heidersdorf.thorsten@gmail.com} 
\address{R. W.: Mathematisches Institut, Universit\"at Heidelberg}
\email{weissaue@mathi.uni-heidelberg.de}

\date{}

\begin{abstract} We study the semisimplification of the full karoubian subcategory generated by the irreducible finite dimensional representations of the algebraic supergroup $GL(m|n)$ over an algebraically closed field of characteristic zero. This semisimplification is equivalent to the representations of a pro-reductive group $H_{m|n}$. We show that there is a canonical decomposition $H_{m|n} \cong GL(m\!-\! n) \times H_{n|n}$, thereby reducing the determination of $H_{m|n}$ to the equal rank case $m\! =\! n$ which was treated in a previous paper. 
\end{abstract}

\subjclass[2020]{17B10, 17B20, 17B55, 18M05, 18M25,  20G05.}

\maketitle

\setcounter{tocdepth}{1}

\section{Introduction}

The categories of algebraic finite dimensional representations $\mathcal{T}_{m|n}$ of the general linear supergroups
$GL(m\vert n)$ over an algebraically closed field $k$ of characteristic zero are abelian tensor categories, but in contrast to the classical $GL(m) = GL(m,0)$-case they are
not semisimple unless $m$ or $n$ is zero. In the following we may assume $r=m-n \geq 0$. The structure of $\mathcal{T}_{m|n}$ as abelian categories is now reasonably well-understood \cite{Serganova-blocks} \cite{Brundan-Stroppel-4} and may be best studied via Khovanov-algebras, certain diagram algebras inspired by constructions of Khovanov in his categorification of tangle invariants. 

\smallskip
The monoidal structure however remains elusive. Partial results, in particular the decomposition of tensor products between any two Kostant modules (covering all irreducible modules of $GL(m|1)$) and all projective modules, were obtained in \cite{Heidersdorf-mixed-tensors} based on earlier work of Brundan-Stroppel \cite{Brundan-Stroppel-5} and Comes-Wilson \cite{Comes-Wilson}. 

\smallskip
Yet the monoidal structure is of crucial interest, foremost for applications in physics, but also with regards to possible versions of geometric super Satake \cite{BF}, applications to Reshetikhin-Turaev type link invariants \cite{Qu} or the hypothetical construction of modular tensor categories from quantized versions of Lie superalgebras.

\smallskip
In this article we will study this problem \emph{up to superdimension zero}, following the approach of the  $m=n$-case in \cite{H-ss} \cite{HW-pieri} \cite{HW-tannaka}.

\subsection{Tensor product decomposition up to superdimension $0$} A very efficient way to think about irreducible representations of $GL(m|n)$ is the language of weight and cup diagrams of \cite{Brundan-Stroppel-4}. We show that each irreducible maximal atypical representation $L(\lambda)$, contained in a block of $ \Tmn$ that is indexed by a central character $\chi_{\lambda}$ with weight/cup diagram say

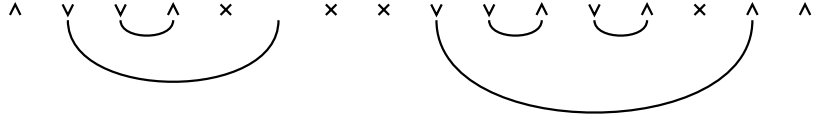
\begin{center}
\smallskip
 
 \scalebox{0.7}{
\begin{tikzpicture}
\foreach \x in {-2,-1,5,6,8} 
     \draw[very thick] (\x-.1, .1) -- (\x,-0.1) -- (\x +.1, .1);
\foreach \x in {-3,0,7,9,11,12} 
     \draw[very thick] (\x-.1, -.1) -- (\x,0.1) -- (\x +.1, -.1);
\foreach \x in {1,3,4,10} 
     \draw[very thick] (\x-.1, .1) -- (\x +.1, -.1) (\x-.1, -.1) -- (\x +.1, .1);

\draw[very thick] [-,black,out=270, in=270](-2,-0.2) to (2,-0.2);
\draw[very thick] [-,black,out=270, in=270](-1,-0.2) to (0,-0.2);
\draw[very thick] [-,black,out=270, in=270](5,-0.2) to (11,-0.2);
\draw[very thick] [-,black,out=270, in=270](6,-0.2) to (7,-0.2);
\draw[very thick] [-,black,out=270, in=270](8,-0.2) to (9,-0.2);


\end{tikzpicture} } \

\end{center}

gives rise to two associated irreducible representations: The first one is an irreducible representation $L(\chi_{\lambda})$, the "classical core" $L(\lambda_{cl})$ of the group $GL(r)$ determined by the position of the $r=m-n$ crosses (the core symbols),

\begin{center}
\smallskip
 
 \scalebox{0.7}{
\begin{tikzpicture}
\foreach \x in {} 
     \draw[very thick] (\x-.1, .1) -- (\x,-0.1) -- (\x +.1, .1);
\foreach \x in {} 
     \draw[very thick] (\x-.1, -.1) -- (\x,0.1) -- (\x +.1, -.1);
\foreach \x in {1,3,4,10} 
     \draw[very thick] (\x-.1, .1) -- (\x +.1, -.1) (\x-.1, -.1) -- (\x +.1, .1);


 node[pos=(-3, 0)]{};
  node[pos=(12, 0)]{};
\path (-3,0)--(12,0);




\end{tikzpicture} }
\smallskip

\end{center}

which depends only on the block of $L(\lambda)$. The second is a representation in ${\mathcal T}_{n\vert n}$,
the "principal core"  $L(\lambda_{pr})$,  obtained under Serganova's block equivalence with the principal block of ${\mathcal T}_{n|n}$ by taking the corresponding irreducible representation  with the cup diagram

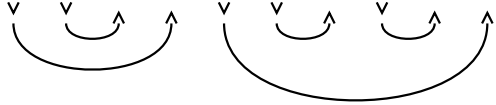
\begin{center}
 
 \scalebox{0.7}{
\begin{tikzpicture}
\foreach \x in {0,-1,3,4,6} 
     \draw[very thick] (\x-.1, .2) -- (\x,0) -- (\x +.1, .2);
\foreach \x in {1,2,8,5,7} 
     \draw[very thick] (\x-.1, -.2) -- (\x,0) -- (\x +.1, -.2);
\foreach \x in {} 
     \draw[very thick] (\x-.1, .1) -- (\x +.1, -.1) (\x-.1, -.1) -- (\x +.1, .1);

\draw[very thick] [-,black,out=270, in=270](0,-0.2) to (1,-0.2);
\draw[very thick] [-,black,out=270, in=270](-1,-0.2) to (2,-0.2);
\draw[very thick] [-,black,out=270, in=270](3,-0.2) to (8,-0.2);
\draw[very thick] [-,black,out=270, in=270](4,-0.2) to (5,-0.2);
\draw[very thick] [-,black,out=270, in=270](6,-0.2) to (7,-0.2); 


\end{tikzpicture} } \ .

\end{center}

For $r=m-n$ and given $L(\lambda)$ this defines as associated object the external tensor product \[ L(\lambda_{cl}) \boxtimes L(\lambda_{pr}) \ \in \ Rep(GL(r)) \times {\mathcal T}_{n|n} \ .\] Here, for algebraic groups $G$ over $k$, $Rep(G)$ 
denotes
 the tensor category of the algebraic representations of $G$ on finite dimensional
$k$-vector spaces. Similarly, $sRep(G)\!=\! Rep(G)\!\otimes\! k^{1\vert 1}$ will denote the tensor category of algebraic representations of $G$ on finite dimensional
super vector spaces over $k$.

\smallskip
As shown in  \cite{HW-tensor}, $(L({\lambda}_{cl}), L(\lambda_{pr}))$ determines the superdimension of $L(\lambda)$ given by \[ \sdim L(\lambda) = \pm m(\lambda) \dim L(\lambda_{cl})\ , \] where $m(\lambda)$ only depends on the cup diagram of $L(\lambda_{pr})$ and $\dim L(\chi_{cl})$ is given by the Weyl dimension formula for $GL(r)$. If the atypicality of $L(\lambda)$ is not maximal, i.e. if it is strictly smaller than $n$, then $\sdim L(\lambda) = 0$.  This assertion however is only the first layer of the following  stronger statement: 
For irreducible maximal atypical representations $L(\lambda)$ and $L(\mu)$, we show in this article that, up to indecomposable modules of superdimension zero, 
the associated
objects $L({\lambda}_{cl}), L(\lambda_{pr})$ and $L({\mu}_{cl}), L(\mu_{pr})$
determine the decomposition of the tensor product $L(\lambda)\otimes L(\mu)$ into indecomposable summands in ${\mathcal T}_{m\vert n}$.

In fact, it reduces this 
to a computation of $L(\lambda_{cl})\otimes L(\mu_{cl})$ and $L(\lambda_{pr})\otimes L(\mu_{pr})$
in the tensor categories $Rep(GL(r))$  and ${\mathcal T}_{n\vert n}$ respectively, i.e. to the case $m=n$ that was  studied in \cite{HW-tannaka}.

\subsection{The main theorems}

\subsubsection{The universal semisimple quotient} To talk about computations up to superdimension 0  in a meaningful way, we divide the tensor category $\mathcal{T}_{m|n}$ by the tensor ideal $\mathcal{N}$ \cite{Andre-Kahn} \cite{BW} of negligible morphisms. The quotient is a semisimple abelian tensor category. 
Furthermore, an indecomposable object is sent to zero if and only if its superdimension is 0. 
Direct sums of such objects are called \emph{negligible}. As observed in \cite{H-ss}, the quotient is a semisimple super-tannakian category, hence isomorphic to the (semisimple) representation category of a certain affine supergroup scheme. We refer to \cite{Andre-Kahn} for the signifance of such semisimplifications in algebraic geometry, to \cite{EAS} for applications in super representation theory and to \cite{Etingof-Ostrik} for further examples coming from representation theory.

\smallskip
Since we are interested in tensor products of irreducible representations, we consider instead of $\mathcal{T}_{m|n}$ the full karoubian subcategory generated by irreducible representations whose objects are all (direct sums of) indecomposable modules which arise in iterated tensor products of irreducible representations. For technical reasons it wil be conveniant and harmless to replace it by the full subcategory $\mathcal{T}_{m|n}^+$ generated by irreducible representations of superdimension $\geq 0$. So, let $\overline{\mathcal{T}}_{m|n}$ denote the corresponding quotient $\mathcal{T}_{m|n}^+/\mathcal{N}$, refered  to as the semisimplification of $\mathcal{T}_{m|n}^+$. Notice that, by Tannakian arguments,  
$\overline{\mathcal{T}}_{m|n}$ itself is the representation category of an a priori unknown supergroup scheme $H_{m|n}$.

\subsubsection{The tannakian tower} The crucial ingredient to determine the groups $H_{m|n}$ is the Duflo-Serganova cohomology functor $DS\!:\!\mathcal{T}_{m|n}\! \to\! \mathcal{T}_{m-1|n-1}$ \cite{Duflo-Serganova} \cite{HW-tensor}.  As shown in \cite{HW-tannaka}, $DS$ induces a symmetric monoidal functor $DS\!:\! \mathcal{T}_{m|n}^+ \!\to\! \mathcal{T}_{m-1|n-1}^+$. This functor gives rise to a symmetric monoidal functor $$d=d_{m|n} \!: \ \overline{\mathcal{T}}_{m|n} \to \overline{\mathcal{T}}_{m-1|n-1}\ .$$ 
Analogous to \cite{HW-tannaka} we have:

\begin{thm} [Theorem \ref{thm:tannaka} and Corollary \ref{cor:fibre}] The categories $\overline{\mathcal{T}}_{m|n}$ are semisimple tannakian categories, hence their Tannaka groups $H_{m|n}$ are projective limits of reductive algebraic groups over $k$. A fibre functor \[\omega: Rep(H_{m|n}) \to vec_k\] is provided by the composite of functors $d_{m|n}$. By Tannakian duality the functor $d_{m|n}$ induces  a closed embedding of affine group schemes $H_{m-1|n-1} \hookrightarrow H_{m|n}$ over $k$ such that $d_{m|n}: Rep(H_{m|n}) \to Rep(H_{m-1|n-1})$ can be identified  with the restriction functor that restricts representations to $H_{m-1|n-1}$ with respect to this group scheme embedding. The composite of functors $d_{m|n}$ induces a functor
 \[\omega: Rep(H_{m|n}) \to Rep(H_{r\vert 0})=Rep(GL(r))\]
\end{thm}

This theorem allowed us in \cite{HW-tannaka} to determine the groups $H_{n|n}$ inductively. The images of indecomposable objects $L$ of $\mathcal{T}_{m|n}^+$  either become an irreducible representation of $\overline{\mathcal{T}}_{m|n}$
or become zero. In the following we consider this for irreducible objects $L$ only. If $\sdim(L(\lambda)) <0$, we will simply replace it by its parity shift. We denote the irreducible representation of positive superdimension with highest weight $\lambda$ by $X_{\lambda}$. Such $X_{\lambda}$ in $\mathcal{T}_{m|n}^+$  generate a tensor subcategory in the quotient isomorphic to $Rep(H_{\lambda})$ for some reductive group $H_{\lambda}$, and the image of $L(\lambda)$ defines a faithful representation $\overline{X}_{\lambda}$ of $H_{\lambda}$. The Tannaka group $H_{m|n}$ is a subgroup of the product of Tannaka groups $H_{\lambda}$
$$ H_{m\vert n} \hookrightarrow \prod_{\lambda/\sim} H_{\lambda} \ ,$$
where $\lambda$ runs over a certain set of equivalence classes $\lambda/\sim$ of maximal atypical highest weights $\lambda$. 
In general $H_\lambda$ may be considered as an algebraic subgroup  of the general linear group $GL(V_\lambda)$ of the finite dimensional $k$-vectorspace $V_\lambda=\omega(X_{\lambda})$ defined by the fibre functor $\omega$. Note that $\dim(V_\lambda) = \sdim(X_{\lambda})$.

\subsubsection{Stabilization} The computation of $H_{m|n}$ and the various groups $H_{\lambda}$ is done by a reduction to the case $m=n$. In the first step we pass to the case where $L(\lambda)$ is \emph{stable} maximal atypical (all atypical hooks $\vee$ are to the left of all $\times$ in the weight diagram). 

Any stable representation $L(\lambda)$ is of the form \[L(\lambda) = L(\lambda_1,\ldots,\lambda_r,\mu_1, \ldots, \mu_n \, | \, - \mu_n, \ldots,-\mu_1)\] where $r=m-n$. We say it is \emph{negatively stable} if $\mu_1 \leq 0$. The Picard group of $\mathcal{T}_{m|n}^+/\mathcal{N}$ contains the determinants of the representations $L(\lambda)$, and in particular $det(V):=\Lambda^r(V)$ for the standard representation $V = k^{m|n}$ of $GL(m|n)$. For $r=m-n$ we define  \[ \Pi := \Lambda^{r}(V) \otimes Ber_{m|n}^{-1} \cong L(0,\ldots,0,-1,\ldots,-1 \, | \, 1,\ldots,1) \ ,\] where we follow the convention of Section  \ref{sec:irreducible} that $Ber_{m|n}$ in $\Tmn$ denotes the Berezin determinant resp. its parity shift, 
depending on whether the superdimension of the Berezin determinant is positive resp. negative.

\begin{thm} (Shifting) For any irreducible representation $L(\lambda)$ in $\mathcal{T}_{m|n}$ the representation $\Pi^\ell \otimes L(\lambda)$ is negatively stable for $\ell >>0$ up to negligible summands.
\end{thm}

Since $\sdim(\Pi)=1$, $\Pi$ defines a character of $H_{m|n}$. Hence tensoring with $\Pi$ is harmless.

\subsubsection{Splicing} We may now assume without loss of generality that $L(\lambda)$ is negatively stable. We separate this into two parts: The classical part and the principal part; irreducible representations in $\mathcal{T}_{m|n}$ which "represent" the classical and principal core. For given $L(\lambda) = L(\lambda_1,\ldots,\lambda_r,\mu_1,\ldots,\mu_n  | \! -\mu_n,\ldots,-\mu_1)$ with $\lambda_r \geq 0$ and $\mu_1 \leq 0$ let \[ \lambda_{pr} = (0,\ldots,0,\mu_1,\ldots,\mu_n  | \! - \mu_n,\ldots,-\mu_1)\] and let \[ \lambda_{cl} = (\lambda_1,\ldots,\lambda_r,0,\ldots,0 \, | \, 0,\ldots,0).\] 

The key statement (that enables the reduction to the case $r=0$, i.e. $m=n$) is the following theorem.

\begin{thm} (Splicing theorem) Any negatively stable irreducible $L(\lambda)$ in ${\mathcal T}_{m\vert n}$ with $\lambda_r \geq 0$ is the tensor product of its classical part with its principal part up to negligible objects of superdimension 0: \[ L(\lambda_{cl})  \otimes L(\lambda_{pr}) \cong 
 L(\lambda_1,\ldots,\lambda_r,\mu_1, \ldots, \mu_n \, | \! - \mu_n, \ldots,-\mu_1) \ \ \text{mod } \mathcal{N}.\]
\end{thm}

This readily generalizes to  maximal atypical $L(\lambda)$ with arbitrary $\lambda_r$ in $\mathbb Z$.

\begin{thm} \label{introthm:dec}(Canonical decomposition) Let  $r=m\!-\! n$ and let $H_{m|n}$ denote the Tannaka group
of the semisimplification of the karoubian tensor envelope of the irreducible representations of superdimension $\geq 0$ in ${\mathcal T}_{m\vert n}$. Then there exists a canonical decomposition of tensor categories \[ Rep (H_{m|n}) \cong Rep (GL(r)) \times \overline{\mathcal T}_{n\vert n} \quad , \quad  \overline{\mathcal T}_{n\vert n} = \mathcal{T}_{n|n}^+/\mathcal{N}.\] For stable $L(\lambda) = L(\lambda_1,\ldots,\lambda_r, \mu \ | \ -\mu)$ and $\mu \in \mathbb{Z}^n$ the image $\overline{X}_{\lambda}$ of $X_{\lambda}$
in the semisimplification, considered as a representation of $H_{m|n} \cong GL(r) \times H_{n|n}$, is of the form \[ \overline{X}_{\lambda} =  L(\lambda_1,\ldots,\lambda_r) \boxtimes \overline{X}_{\mu} \] for the irreducible highest weight representation   $L(\lambda_1,\ldots,\lambda_r)$ of $GL(r)$ and 
the irreducible faithful maximal atypical representation $X_{\mu}\cong L(\mu \, |  \! -\mu)$ of $GL(n|n)$. In particular, we have  $H_{\lambda} \cong GL(r) \times H_{\mu}$.
\end{thm}

The various  groups $H_{\mu}$ that can occur have been (almost) determined in \cite{HW-tannaka}. They are given by classical groups such as $GL(d)$, or $SO(d)$ and $Sp(d)$ for $d=\sdim(X_{\mu})$, where $\overline{X}_{\mu}$ is isomorphic either to the standard representation, its dual or a character twist of these, except for certain exceptional cases. 
Here the caveat is that for certain selfdual representations $L(\mu  |\!  -\mu)$ in $\mathcal{T}_{n|n}^+$ the groups $H_{\mu}$ were not fully determined, the smallest example being the representation $L(3,2,1,0 \, |\,  0,-1,-2,-3)$ of $GL(4|4)$. In these exceptional cases (that conjecturally should not exist) the derived connected component
of $H_{\mu}$ could also be the group $SL(\frac{d}{2})$.

\medskip
Since the fusion rules are known for $GL(r)$, and are also known for the various $H_{\lambda}$ that can occur, we have indeed obtained a (partial) description of the decomposition law of tensor products of irreducible representations into indecomposable modules up to negligible indecomposable summands; and moreover a classification of the indecomposable modules of non-vanishing superdimension in $\mathcal{T}_{m|n}^+$. In down to earth terms, the theorem says that the fusion rules are up to superdimension zero given by a mixture of the Littlewood-Richardson rule (which determines the blocks of the appearing summands) and the fusion rules of the $H_{\lambda}$ (which depend only on the $m\!=\! n$-case). It is intriguing to ask whether one can iterate this construction in some way and understand the category $\Tmn$ successively by considering its atypical layers.

\subsubsection{Proof of the main theorem} Recall $r=m-n$. We can realize the classical group $H_{cl}= GL(r)$ in two ways: as the image of the enriched fibre functor \[ \omega: \overline{\mathcal{T}}_{m|n}^+ \to Rep(GL(r)) \] and as the full tensor subcategory \[i: Rep(GL(r)) \to \overline{\mathcal{T}}_{m|n}^+\] generated by the image of the standard representation in the semisimplification. The composition $\omega \circ i$ defines an autoequivalence of $Rep(GL(r))$ which induces a splitting \[ H_{m|n} \cong GL(r) \times H_{pr}\] with some unknown group $H_{pr}$. The canonical decomposition (Theorem \ref{introthm:dec}) implies $H_{pr} = H_{n|n}$. We can easily reduce this theorem to the equivalent claim that the images ${\mathcal P}_{m\vert n}$ resp. ${\mathcal P}_{n\vert  n}$ of the principal blocks in the semisimplifications are equivalent as tensor categories. We first consider the negative halfs of the principal blocks: the intersection of $\mathcal{T}_{m|n}^+$ with the full abelian subcategory generated by irreducible objects \[ X_{\lambda}, \ \lambda = (0,\ldots,0,\mu_1,\ldots,\mu_n  | \! - \mu_n,\ldots,-\mu_1)\] (the $m\! >\! n$-case) or the full abelian subcategory generated by irreducible objects $X_{\mu}, \ \mu=(\mu_1,\ldots,\mu_n \ | \ - \mu_n,\ldots,-\mu_1)$ with $\mu_1 \leq 0$. To show ${\mathcal P}_{m\vert n} \cong {\mathcal P}_{n\vert  n}$, it is enough to show that the images of the negative halfs are equivalent. Indeed the Tannakian Reconstruction Lemma \ref{reconstruction} shows that the Tannaka group $H_{pr}$ of ${\mathcal P}_{m\vert n}$ can be reconstructed from the Tannaka monoid of the negative half. To finally show that the images of the two negative halfs are tensor equivalent, we consider the functor $\eta = ()^{GL(r)} \circ res$ $$ \eta:  \mathcal{T}_{m|n}^+ \xrightarrow{res}  Rep(GL(r) \times GL(n|n)) \xrightarrow{(\ )^{GL(r)}}  Rep(GL(n|n))\ , $$ $$ \eta(X)   = Hom_{GL(r)}(\one, res(X))\ . $$ We show that this functor coincides with Serganova's block equivalence if restricted to the principal block ${\mathcal B}_{\one}$ in $\Tmn$, and it is monoidal when restricted to the negative half of the principal block. This allows us to show  ${\mathcal P}_{m\vert n} \cong {\mathcal P}_{n\vert  n}$ in Theorem \ref{thm:principal}.

\subsubsection{Determinants}

Independent of the exact determination of the groups $H_{\mu}$, we can always compute the determinant $\det(X_{\lambda}) := \Lambda^{\sdim(X_{\lambda})} (X_{\lambda})$ (up to negligible summands) of a maximal atypical irreducible representation. One possible proof uses the Canonical Decomposition Theorem \ref{introthm:dec} to reduce the problem to the principal block, where an argument similar to the one given in \cite{HW-tannaka} expresses the determinant as a power of the first ground state of the block. However, the canonical decomposition in fact allows to reduce the problem to the $m=n$-case in a straightforward manner.

\begin{thm} Let $X_{\lambda} \in \mathcal{T}_{m|n}^+$ with image \[ L(\chi_{\lambda}) \boxtimes \overline{X}_{\mu} \in \Rep(GL(r) \times H_{n|n}).\] Then \[ \det(X_{\lambda}) = \det( L(\chi_{\lambda}))^{\dim \overline{X}_{\mu} } \boxtimes \det(\overline{X}_{\mu})^{\dim L(\chi_{\lambda})}\] where $\det(\overline{X}_{\mu})$ is represented by a Berezin power in $\mathcal{T}_{n|n}^+$ (which can be expressed in terms of the cup diagram) by \cite[Theorem 14.3]{HW-tannaka}. 

\end{thm}

\subsubsection{Physical motivation} As explained in \cite[Section 1, Section 15]{HW-tannaka} the supergroups $GL(m|4)$ naturally occur  in supersymmetric field theories. The real supergroups $G=SU(2,2\vert N)$ are covering groups of
the super conformal groups $SO(2,4\vert N)$, and the complexification ${\frak g}$ of $Lie(G)$ is isomorphic to the complex Lie superalgebras $\mathfrak{sl}(4\vert N)$. Hence the fusion rules for representations of $GL(N|4)$ are required for the computation of Feynman integrals if we consider  supersymmetric fields $\psi$ on Minkowski space (or a covering space of a conformal invariant compactification) with values in a finite dimensional representation $L$ of ${\frak g}$. Similarly to \cite{HW-tannaka} we  raise the question whether the groups $H_{m|n}$ or $H_{\lambda}$ may be interpreted as internal symmetry groups of such theories in an approximate sense, i.e. whether  physical observers might come up with the impression that the underlying rules of symmetry are imposed by the hidden invariant theory of the quotient tensor category $\overline{\mathcal T}_{m|n}\! =\! Rep(H_{m|n})$ instead of the invariant theory of $\mathcal T_{m|n}$ (a phenomen that should be closely related to supersymmetric cancellations).

\subsection*{Acknowledgements}
The research of T.H. was partially funded by the Deutsche Forschungsgemeinschaft (DFG, German Research
Foundation) under Germany's Excellence Strategy – EXC-2047/1 – 3270685813.



\section{The category $\mathcal{T}_{m|n}^+$}

\subsection{The category $\mathcal{T} = \mathcal{T}_{m|n}$} Let $k$ be an algebraically closed field of characteristic zero. We adopt the notations of \cite{HW-tensor}. With $GL(m|n)$ we denote the general linear supergroup and by $\mathfrak{g} = \mathfrak{gl}(m|n)$ its Lie superalgebra. A representation $\rho$ of $GL(m|n)$ is a representation of $\mathfrak{g}$ on a finite dimensional $k$ vector space such that its restriction to $\frakg_{\bar{0}}$ comes from an algebraic representation of $G_{\bar{0}} = GL(m) \times GL(n)$. Let $\mathcal{T} = \mathcal{T}_{m|n}$ denote the category of all finite dimensional representations with parity preserving morphisms.

\medskip
As an abelian category $\Tmn$ splits into a direct sum $\Tmn\!=\! \bigoplus_{\mathcal B} {\mathcal B}$ of abelian subcategories $\mathcal B$, the so called blocks of $\Tmn$. The principal block, the block containing the trivial representation $\mathbf 1$, will be denoted ${\mathcal B}_{\mathbf 1}$.
Although not stable under the tensor product, ${\mathcal B}_{\mathbf 1}$ will be of particular importance later.

\subsection{Irreducible representations} \label{sec:irreducible} 
Let $\mbox{\textsf{par}}$ denote the parity shift in $\Tmn$.  
Up to parity shift the irreducible representations  in $\Tmn$ are  parametrized by their highest weight with respect to the Borel subalgebra of upper triangular matrices. A weight $\lambda=(\lambda_1,...,\lambda_m \ | \ \lambda_{m+1}, \cdots, \lambda_{m+n})$ of an irreducible representation in $\Tmn$ satisfies $\lambda_1 \geq \ldots \geq \lambda_m$, $\lambda_{m+1} \geq \ldots \geq \lambda_{m+n}$ with integer entries. In particular the irreducible representations in $\mathcal{T}_{n}$ are given by the $\{L(\lambda), \mbox{\textsf{par}}\, L(\lambda) \ | \ \lambda \in X^+ \}$ where $X^+$ denotes the set of dominant integral weights and $\mbox{\textsf{par}}$ denotes the parity shift. The Berezin determinant of the supergroup $GL(m|n)$ defines a one dimensional representation. Its weight is given by $\lambda_i=1$ and $\lambda_{m+i}=-1$ for $i=1,..,n$. Its superdimension is $(-1)^n$.

\begin{definition} For maximal atypical $\lambda$ we write $X_{\lambda} = \mbox{\textsf{par}}^{p(\lambda)} L(\lambda)$ for $p(\lambda) = \sum_{i=1}^n \lambda_{m+i}$. In other words, \[ X_{\lambda} = \begin{cases} L(\lambda) & \sdim(L(\lambda)) > 0 \\ \mbox{\textsf{par}} L(\lambda) & \sdim(L(\lambda)) < 0.\end{cases}\] The only exception to this rule is the Berezin determinant where we will use the notation $Ber_{m\vert n}$ to denote the $(-1)^n$-fold parity shift of the Berezin determinant 
in ${\mathcal T}_{m\vert n}$. Thus $Ber_{m\vert n}$ has always superdimension $1$ and $Ber_{m\vert n}=L(1,...,1\vert -1,...,-1)$.
\end{definition}

\subsection{Dualities} There are two different dualities on $\mathcal{T}_{m|n}$, the usual dual $V^\vee=(V^\vee,\rho^\vee)$ and the twisted dual $V^* =(V^{\vee},\rho^\vee \circ \tau) $ where $\tau(x)=-x^T$ is the automorphism of $\g$ defined by the supertranspose on $\g$. The twisted dual stabilizes irreducible and projective modules.

\subsection{The karoubian category $\mathcal{T}_{m|n}^+$}

Let $\mathcal{T}_{m|n}^+$ denote the full subcategory of $\mathcal{T}_{m|n}$, whose objects consist of all retracts of iterated tensor products of irreducible representations in $\mathcal{T}_n$ that are not maximal atypical and of  maximal atypical irreducible representations $L(\lambda)$ of positive superdimension. The category $\mathcal{T}_{m|n}^+$ is a rigid symmetric monoidal  idempotent complete $k$-linear category closed under the $*$-involution. It contains all irreducible objects of $\mathcal{T}_{m|n}$ up to a parity shift.

\begin{remark} We could also work with the karoubian subcategory generated by all irreducible representations.  However it is more convenient for tannakian arguments to work with a category where all objects have categorical dimensions $\geq 0$.
\end{remark}

\begin{example} Kac modules $V(\lambda)$ of 
${\mathcal T}_{m|n}$ are not contained in ${\mathcal T}_{m|n}^+$ for atypical $\lambda$. More generally, if $X \in \mathcal{T}_{m|n}^+$ and $X$ has a filtration by Kac or anti-Kac objects, then $X$ is projective (see \cite{HW-tannaka}).
\end{example}

\subsection{Weight diagrams and cups} \label{section: weights} 

Let $\lambda \in X^+$. Following \cite{Brundan-Stroppel-4} to each highest weight $\lambda$ in $X^+(m|n)$  we associate  two subsets of cardinality $m$ respectively $n$ of the numberline $\mathbb Z$
\begin{align*} I_\times(\lambda)\ & =\ \{ \lambda_1  , \lambda_2 - 1, ... , \lambda_m - m +1 \} \\
 I_\circ(\lambda)\ & = \ \{ 1 - m - \lambda_{m+1}  , 2 - m - \lambda_{m+2} , ... ,  n-m - \lambda_{m+n}  \} \ ,  \end{align*}
and define a labeling of the numberline $\mathbb Z$:
The integers in $ I_\times(\lambda) \cap I_\circ(\lambda) $ are labeled by $\vee$, those in $I_\times(\lambda)$ respectively $I_\circ(\lambda)$ are labeled by $\times$ respectively $\circ$. The remaining integers are labeled by $\wedge$. 
This labeling of the numberline uniquely characterizes the weight vector $\lambda$. The block of $L(\lambda)$ is uniquely characterized by the position of the crosses and circles. The number of times the label $\vee$ occurs in the labeling is called the {\it degree of atypicality} $atyp(\lambda)$ of $\lambda$. Notice $0 \leq atyp(\lambda) \leq n$, and for $atyp(\lambda) =n$ the weight $\lambda$ is called
{\it maximal atypical}. For $GL(n|n)$ a weight is maximally atypical if and only if $\lambda_i = - \lambda_{n+i}$ for $i=1,\ldots,n$ in which case we write $$L(\lambda) = [\lambda_1,\ldots,\lambda_n]\  .$$ All maximal atypical weights of $GL(n|n)$ lie in the principal block ${\mathcal B}_{\one}$ and the weight diagrams do not contain any $\times$'s or $\circ$'s.

\medskip
To each weight diagram one associates a cup diagram as in \cite{Brundan-Stroppel-4}. The outer caps in a cup diagram define the sectors of the weight as in \cite{HW-tensor}. We number the sectors from left to right $S_1$, $S_2$, $\ldots$, $S_k$. Each sector $S_i$ defines an interval $[a_i,b_i]$ on the numberline where the leftmost $\vee$ of the sector occurs at position $a_i$. The $\times$'s and $\circ$'s do not play any role in our arguments. They are fixed and otherwise ignored unless specifically mentioned.

\begin{example} The following shows a cup diagram with 3 sectors.
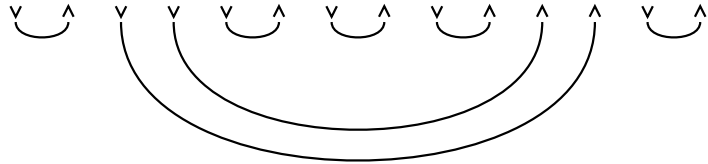
\begin{center}
\medskip
 
 \scalebox{0.7}{
\begin{tikzpicture}
\foreach \x in {-6,-4,-3,-2,0,2,6} 
     \draw[very thick] (\x-.1, .1) -- (\x,-0.1) -- (\x +.1, .1);
\foreach \x in {-5,-1,1,3,4,5,7} 
     \draw[very thick] (\x-.1, -.1) -- (\x,0.1) -- (\x +.1, -.1);
%


\draw[very thick] [-,black,out=270, in=270](-6,+-0.2) to (-5,+-0.2);
\draw[very thick] [-,black,out=270, in=270](-4,+-0.2) to (5,+-0.2);
\draw[very thick] [-,black,out=270, in=270](-3,+-0.2) to (4,+-0.2);
\draw[very thick] [-,black,out=270, in=270](-2,+-0.2) to (-1,+-0.2);
\draw[very thick] [-,black,out=270, in=270](0,+-0.2) to (1,+-0.2);
\draw[very thick] [-,black,out=270, in=270](2,+-0.2) to (3,+-0.2);
\draw[very thick] [-,black,out=270, in=270](6,-0.2) to (7,-0.2);


\end{tikzpicture} }
\end{center}
\end{example}

By \cite{HW-tensor} the cup diagrams of irreducible modules in the principal block of $GL(n|n)$ uniquely correspond  to spaced forests of rank $n$ (by interpreting the outer cup of a sector as the root of a tree). For example, the cup diagram above corresponds to the spaced forest 
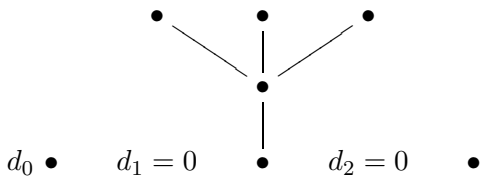
\[ \xymatrix@R-0.3cm@C-0.3cm{ & \bullet \ar@{-}[dr]  & \bullet \ar@{-}[d] & \bullet \ar@{-}[dl] & \\ & &  \bullet \ar@{-}[d] & & \\
d_0  \ \bullet & d_1 = 0 & \bullet & d_2 = 0 & \bullet } \] where $d_1$ and $d_2$ are the distances between sectors $S_1$ and $S_2$ and between $S_2$ and $S_3$ and $d_0$ depends on the position on the numberline

For arbitrary $L(\lambda) \in \mathcal{T}_{m|n}$ we need to incorporate the $\times$'s and $\circ$'s. This leads to the notion of a marked spaced forest. By definition, such marked spaced forests ${\calF}$ are defined by data $$(d_0,{\calT}_1,d_1,{\calT}_2, \cdots, d_{k-1},{\calT}_k, \{a_1,\ldots,a_d\}, \{b_1,\ldots,b_{d'}\})$$ where the $\calT_i$ for $i=1,...,k$  are rooted planar trees positioned on points of the numberline  from left to right and the sets  $\{a_1,\ldots,a_d\}$, $\{b_1,\ldots,b_{d'}\}$ (encoding the crosses and circles) comprise of points on the numberline.

\subsection{Tannaka duality} The following generalizes \cite[Proposition 21.1]{HW-tensor}.

\begin{prop} \label{irreducible-dual} For $r\!=\! m\!-\! n$  and irreducible $L$  in $\mathcal{T}_{m|n}$ the weight diagram of the dual $L^\vee$ is obtained from the weight diagram of $L$ as follows: Interchange all $\vee \wedge$-pairs in cups, then apply the reflection $s \mapsto 1- r - s$ to each symbol.
\end{prop} 

 The proof is more or less the same as in \cite{HW-tensor}: We dualize $P(\lambda)$ which is easier since $P(\lambda)$ occurs as a direct summand in a mixed tensor space $V^{\otimes i} \otimes (V^{\vee})^{\otimes j}$ for some $i,j$; and for such summands there is a closed diagrammatic description of the dual \cite{Heidersdorf-mixed-tensors}.

\medskip
In terms of the associated marked spaced forest the description is as follows.

\begin{lem} The weight of the dual representation corresponds to the spaced forest \[ \calF^{\vee} = (d_0^*, \calT_k^*, d_1^*, \calT_{k-1}^*, d_2^*, \ldots, d_{k-1}^*, \calT_1^*, \{a_1',\ldots,a_d' \}, \{b_1',\ldots,b_{d'}' \})\] where $d_i^* := d_{k-i}$ for $i=1,\ldots,k-1$ and $d_0^*  =   - d_0 - d_1 - \ldots - d_{k-1}$, \[a_i' = 1-r -a_i, \ \ b_i' = 1- r - b_i \]  and $\calT_i^{*}$ denotes the mirror image (along the root axis) of the planar tree $\calT_i$.
\end{lem}



\section{The DS functor}

\subsection{The Duflo-Serganova functor} \label{ds-functor} For maximal atypical $L(\lambda)$ we put $\varepsilon(L(\lambda)) = (-1)^{p(\lambda)}$ for the parity $p(\lambda) = \sum_{i=1}^n \lambda_{m+i}$. If $L(\lambda)$ has atypicality $k$, we use the parity of the corresponding irreducible representation in the principal block of $\mathfrak{gl}(k|k)$.

\medskip

Fix the following element $x\in \mathfrak{g}_1$, \[ x = \begin{pmatrix} 0 & y \\ 0 & 0 \end{pmatrix} \text{ for } \ y = \begin{pmatrix} 0 & 0 & \ldots & 0 \\ 0 & 0 & \ldots & 0 \\ \ldots & & \ldots &  \\ 1 & 0  & 0 & 0 \\ \end{pmatrix}. \]  Then $\rho(x)^2 =0 $ for any representation $(V,\rho)$ of $GL(m|n)$ and we define the cohomological tensor functor $DS$ as  \[ DS =  DS_{m|n,m-1|n-1}: \mathcal{T}_{m|n} \to \mathcal{T}_{m-1|n-1} \]
via  $DS_{m|n,m-1|n-1}(V,\rho)= V_x:=Kern(\rho(x))/Im(\rho(x))$ \cite{Duflo-Serganova}\cite{HW-tensor}.

\begin{thm} \label{mainthm} \cite[Theorem 16.1]{HW-tensor} Let $L(\lambda)$ be an irreducible
atypical representation with cup diagram $$ \bigcup_{j=1}^k \ \ [a_j,b_j] $$ with $k$ sectors
$[a_j,b_j]$. Then $DS(L(\lambda))$ is a direct sum 
$$DS(L(\lambda)) \ \cong\ \bigoplus_{i=1}^k  \ \mbox{\textsf{par}}^{n_i} (L(\lambda_i))$$ for  irreducible atypical
representations $L(\lambda_i)$ in $\mathcal{T}_{m-1|n-1}$ with shift $n_i \equiv \varepsilon(\lambda)
- \varepsilon(\lambda_i)$ mod 2. All $L(\lambda_i)$ are in the block specified by the crosses and circles of $\lambda$, so $L(\lambda_i)$ is uniquely defined by
the property that its cup diagram is given by the union $$ [a_i +1, b_i-1] \ \ \ \cup \ \ \bigcup_{j=1, j\neq i}^k \ \ [a_j,b_j] \ $$  of the sectors in the segment $[a_i+1,b_i-1]$ and the sectors $[a_j,b_j]$ for $1\!\leq\! j \neq i\! \leq\! k$.
\end{thm}

\noindent  
In down to earth terms, the $i$-th summand $L(\lambda_i)$ is obtained by removing the $i$-th outer cup along with its $\vee$. It follows that $DS(L(\lambda))$ is semisimple and multiplicity free. 

\begin{proof} Theorem \ref{mainthm} is proven in \cite{HW-tensor} for the case $GL(n|n)$, however works up to notational changes in the $GL(m|n)$-case. This needs the rules for translation functors acting on irreducible modules, which admit exactly the same description for $GL(m|n)$ as for $GL(n|n)$ \cite[Lemma 2.4]{Brundan-Stroppel-4} since any combinatorial rule involving the weight diagrams is identical once the block labels $\times$ and $\circ$ have been fixed. Alternatively one can reduce the proof to the principal block of $GL(n|n)$, using remarks on block equivalences as in \cite[Section 16]{HW-tensor}, see also \cite[Section 6]{GH} for details.
\end{proof}

\begin{example} \label{example:DS-S-i} Consider $L(\lambda) = L(0,\ldots,0,-i \ | \ i,0)$ in $\mathcal{T}_{m|2}$. Here is the cup diagram for $i=5$ and $m=4$

\medskip
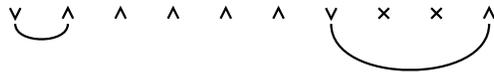
\begin{center}
 
 \scalebox{0.7}{
\begin{tikzpicture}
\foreach \x in {-8,-2} 
     \draw[very thick] (\x-.1, .1) -- (\x,-0.1) -- (\x +.1, .1);
\foreach \x in {-7,-6,-5,-4,-3,1} 
     \draw[very thick] (\x-.1, -.1) -- (\x,0.1) -- (\x +.1, -.1);
\foreach \x in {0,-1} 
     \draw[very thick] (\x-.1, .1) -- (\x +.1, -.1) (\x-.1, -.1) -- (\x +.1, .1);
%


\draw[very thick] [-,black,out=270, in=270](-2,+-0.2) to (1,+-0.2);
\draw[very thick] [-,black,out=270, in=270](-8,+-0.2) to (-7,+-0.2);


\end{tikzpicture} }
\smallskip

\end{center}

The two irreducible summands of $DS(L) \in \mathcal{T}_{m-1|1}$ obtained from the two sectors give

\begin{align*} DS(L(0,\ldots,0,0,-i \ | \ i,0)) & \ \cong\  \mbox{\textsf{par}}\, \one \, \oplus\,  \mbox{\textsf{par}}\, L(0,\ldots,0,0,-i-1 \ | \ i+1).\end{align*}

\end{example}

\subsection{The core}  Applying $DS$ to an irreducible representation $L(\lambda)$ of atypicality $k$ repeatedly yields an isotypic typical representation of $GL(m-k|n-k)$. This irreducible representation is called the classical core $L(\chi_{\lambda})$ of $L(\lambda)$, denoted by Serganova as $L^{core}$ \cite{Serganova-kw}. Up to isomorphism $L(\chi_\lambda)$ depends only on the block of $L(\lambda)$. Its weight diagram is determined by the $\times$'s and $\circ$'s, given by the central character $\chi_{\lambda}$ of $L(\lambda)$. Its multiplicity is a number between $1$ and $n!$ \cite{HW-tensor}. If $L(\lambda)$ is maximal atypical, $L(\chi_{\lambda})$ is a representation of $GL(r)$, and this defines a bijection between maximal atypical blocks of $GL(m|n)$ and irreducible representations of $GL(r)$ for $r=m-n$.



\section{Semisimplification and the tannakian tower}

\subsection{Conventions} \label{Conventions} Our definition of a tensor category is the one used in \cite{EGNO} except that we do not require the category to be abelian or rigid. In particular, they are $k$-linear symmetric monoidal
categories with unit object $\mathbf 1$ and $End(\one) = k$. 
Tensor functors are as in \cite{EGNO}, but need not be exact. Hence our definition of a tensor functor agrees with the one  in \cite{Deligne-Milne}. We say a tensor category $\mathcal{T}$ is tannakian if the following holds:

\begin{enumerate} 
\item[(i)] $\mathcal{T}$ is abelian
\item [(ii)] the canonical morphism $k \to \End(\one)$ is an isomorphism
\item [(iii)] $\mathcal{T}$ is rigid
\item [(iv)] every object in $\mathcal{T}$ has finite length
\item [(v)] there exists a fibre functor $\omega\!:\! \mathcal{T}\! \to\! k$, i.e. an exact faithful $k$-linear tensor functor.
\end{enumerate}
Note that finite length implies that every morphism space is finite dimensional \cite[Prop. 1.1]{Deligne-tensorielles}. Our tensor categories are therefore locally finite in the sense of \cite[Def. 1.8.1]{EGNO}, and the Jordan-H\"older Theorem and Krull-Schmidt Theorem hold.

\subsection{The ideal of negligible morphisms} For any two objects $A, B \in \mathcal{T}_{m|n}^+$ we define $\mathcal{N}(A,B) \subset Hom(A,B)$ by \[ \mathcal{N}(A,B) = \{ f \in Hom(A,B) \ | \ \forall g \in Hom(B,A), \ Tr(g \circ f ) = 0 \}. \] The collection of all $\mathcal{N}(A,B)$ defines a tensor ideal $\mathcal{N}$ of $\Tmn$ \cite{Andre-Kahn}, the tensor ideal of negligible morphisms. By \cite{H-ss} the quotient category $\Tmn/\calN$ is a semisimple super tannakian category. The proofs of the results in \cite[Section 5]{HW-tannaka} carry over verbatim since they are either general tannakian statements or only use that $DS(L(\lambda))$ is semisimple and the sign rule for the parity shifts. We summarize the main results.

\begin{thm} \label{thm:tannaka} (Positivity and the tannakian tower) 
i) $\mathcal{T}_{m|n}^+/\calN$ is a semisimple tannakian category $\overline{\mathcal{T}}_{m|n}:=\mathcal{T}_{m|n}^+/\calN$, i.e. it is tensor equivalent to the representation category of a pro-reductive group scheme $H_{m|n}$.

(ii) $DS$ restricts to a functor $DS: \Tmn^+ \to \mathcal{T}_{m-1|n-1}^+$ which gives rise to a $k$-linear exact tensor functor
between the quotient categories
$$  d:  \overline{\mathcal{T}}_{m|n} \to \overline{\mathcal{T}}_{m-1|n-1} \ .$$

iii) The functor $d$ induces an injective homomorphism of affine $k$-groups
$$  f: H_{m-1|n-1} \longrightarrow H_{m|n} $$ and the functor $ Rep(H_{m|n})\to Rep(H_{m-1|n-1})$ induced  
by $d\!:\! \mathcal{T}_{m|n}^+ \!\to\! \mathcal{T}_{m-1|n-1}^+$ can be identified with the restriction
functor $f^*$ for the homomorphism $f$.
\end{thm}

In other words, the description of the functor $DS$ on irreducible objects in $\mathcal{T}_{m|n}$,
given by Theorem \ref{mainthm}, can be interpreted as branching rules for the inclusion
$$ f: \ H_{m-1|n-1} \hookrightarrow H_{m|n}\ .$$ Note that the functor $d$ on objects is described by $DS$  as in \cite[Section 5]{HW-tannaka} (and was denoted $\eta$ in \cite{HW-tannaka}). Iteration defines the tensor functor
$$  \omega = d^n: \overline{\mathcal{T}}_{m|n} \to Rep(GL(r))  \ .$$

\begin{cor} \label{cor:fibre} For the (forgetful) fibre functor $\omega_r\!:\! Rep(GL(r))\!\to\! vec_k$, the composed tensor functor $\omega_{m\vert n} =  \omega_r \circ \omega :\overline{\mathcal{T}}_{m|n} \to vec_k$ defines a fibre functor for $\overline{\mathcal{T}}_{m|n}$. 
\end{cor}

Similarly to \cite[Section 5.5]{HW-tannaka} one can construct an enriched group homomorphism
$$  f^\bullet:\  H_{m-1|n-1} \times \mathbb G_m   \to H_{m|n} \ $$ using a graded version of the $DS$ functor. Its restriction to the subgroup $1 \times H_{m-1|n-1}$ is the homomorphism $f$ 
from above. Iterating this construction yields a group homomorphism $$h: (\mathbb G_m)^n \to H_{m|n} \ .$$

\begin{remark} As in \cite[Remark 5.7, Lemma 5.8]{HW-tannaka} all objects $X$ of $\mathcal{T}_{m|n}^+$ satisfy condition {\tt T} in the sense that $\overline\partial$ is trivial on $DS(X)$ \cite[Proposition 8.5]{HW-tensor}. As a consequence the Dirac functor functor $H_D(.)$ of \cite{HW-tensor} is naturally equivalent to the
functor $DS: \mathcal{T}_{m|n}^+ \to \mathcal{T}_{m-1|n-1}^+$. 
\end{remark}

\subsection{Reconstruction}

Let $\mathcal T$ be a  $k$-linear tannakian category $\mathcal T$
with a fibre functor $\omega_{\mathcal T}: {\mathcal T} \to vec_k$, i.e. a $k$-linear exact faithful monoidal functor. 
Then, by tannakian duality, $\mathcal T$ as a $k$-linear tannakian category is equivalent
to the category $Rep(G)$ of an affine group scheme $G$. 
Recall that over an algebraically closed field $k$ of characteristic 0 the group $G(k)$ 
determines the Tannaka group $G=G({\mathcal T},\omega_{\mathcal T})$ as an algebraic group over $k$. The elements of $G(k)$ are given
by the invertible monoidal  natural transformations $\gamma$ of the fibre functor $\omega_{\mathcal T}$.
This means: For $X,Y\!\in\! \mathcal T$ the transformations $\gamma_X: \omega_{\mathcal T}(X)\! \to\!  \omega_{\mathcal T}(X)$
are $k$-linear functorial morphisms such that $\gamma_{X\otimes Y}: \omega_{\mathcal T}(X\otimes Y) \to  \omega_{\mathcal T}(X\otimes Y)$ satisfies $\gamma_{X\otimes Y}= \gamma_X\otimes \gamma_Y$ with respect to  
the given isomorphisms $\omega_{\mathcal T}(X\otimes Y) \cong \omega_{\mathcal T}(X)\otimes \omega_{\mathcal T}(Y)$
that are part of the data defining the fibre functor $\omega_{\mathcal T}$; furthermore it is required that
$\gamma_X$ is multiplication by 1 for the unit object $X=\mathbf 1$. 
That $\gamma_X$ is functorial furthermore requires that $\gamma_X \circ  f = f \circ \gamma_Y$ holds for all morphisms
$f:X\to Y$ of the category $\mathcal T$. For details see \cite[Page 129]{Deligne-Milne}.

\medskip
Let then $\mathcal M$ be a $k$-linear monoidal subcategory of $\mathcal T$ with
the following properties:
\begin{enumerate}
\item $\mathcal M$ is a full abelian subcategory of $\mathcal T$ and contains the unit object $\mathbf 1$.
\item There exists an invertible object $\mathcal L$ in $\mathcal M$
with the following property: For every object $X$ in $\mathcal T$ there exist an integer
$k$ such that ${\mathcal L}^{k} \otimes X$ is in $\mathcal M$.
\end{enumerate}

The restriction of $\omega_{\mathcal T}$ to the subcategory $\mathcal M$ defines an exact $k$-linear
monoidal functor $\omega_{\mathcal M}: {\mathcal M} \to vec_k$. The Tannaka monoid
of $(\mathcal M, \omega_{\mathcal M})$ consists of all  invertible monoidal natural transformations $\gamma$ of the fibre functor $\omega_{\mathcal M}$ that respect
the tensor product of $\mathcal M$ similar as above; see \cite[Page 131ff.]{Deligne-Milne}. Under the above assumptions (1) and (2) on $\mathcal M$,
the following holds:

\begin{lem} \label{reconstruction} The Tannaka monoid of $({\mathcal M}, \omega_{\mathcal M})$ coincides with the
Tannaka group of $({\mathcal T}, \omega_{\mathcal T})$. The Tannaka group $G({\mathcal T},\omega_{\mathcal T})$
can be reconstructed from the data $({\mathcal M}, \omega_{\mathcal M}, {\mathcal L})$.
\end{lem}

\begin{proof}
Obviously any monoidal natural transformation $\gamma$ in $G({\mathcal T},\omega_{\mathcal T})(k)$ restricts to
a monoidal natural transformation in the Tannaka monoid of $({\mathcal M}, \omega_{\mathcal M})$.
It only remains to show that conversely any $k$-rational point in the Tannaka monoid of $({\mathcal M}, \omega_{\mathcal M})$, considered as a monoidal natural transformation $\gamma$  of $\omega_{\mathcal M}$, 
can be extended to a monoidal natural transformation of $\omega_{\mathcal T}$, i.e. to a point in $G(k)$ for $G=G({\mathcal T},\omega_{\mathcal T})$.

\medskip
Notice that $\gamma: \omega_{\mathcal M}({\mathcal L}) \to \omega_{\mathcal M}({\mathcal L})$,
as a $k$-linear automorphism of the one-dimensional $k$-vectorspace $\omega_{\mathcal M}({\mathcal L})$, is multiplication by a uniquely defined element $\lambda\in k^*$.
Since for any $X$ in $\mathcal T$  by property (2) there exists $X({\mathcal M})$ in $\mathcal M$ such that $X = {\mathcal L}^{-k} \otimes X({\mathcal M})$, we may define $\gamma_X: \omega_{\mathcal T}(X)\to \omega_{\mathcal T}(X)$ by
$$\gamma_X:=  \lambda^{-k} \cdot \gamma_{X({\mathcal M})} \ .$$ 
This definition does not depend on the choice of $k$ since $X = {\mathcal L}^{-k'} \otimes Y({\mathcal M})$ 
for $Y({\mathcal M})$ implies  $\lambda^{-k'}  \gamma_{Y({\mathcal M})}
= \lambda^{-k} \gamma_{X({\mathcal M})}$. We may assume $\ell = k' - k \geq 0$. Then $Y({\mathcal M}) =  {\mathcal L}^{\otimes \ell} \otimes X({\mathcal M})$ holds, and our claim 
 follows from  $\gamma_{Y({\mathcal M})}=
 \gamma_{{\mathcal L}}^\ell \otimes \gamma_{X({\mathcal M})}$ and $\gamma_{{\mathcal L}} = \lambda$.
 
\medskip
For any $X,Y$ in $\mathcal T$ we find a common $k$ such that ${\mathcal L}^{-k} \otimes X$ and ${\mathcal L}^{-k}\otimes Y$
are in $\mathcal M$. For morphisms $f:X\to Y$ then (1) implies  that the extension of $\gamma$ is functorial, and
 $\gamma_{X\otimes Y}=\gamma_X\otimes \gamma_Y$ is an immediate consequence of   $\gamma_{X({\mathcal M})\otimes Y({\mathcal M})}=\gamma_{X({\mathcal M})}\otimes \gamma_{Y({\mathcal M})}$.
\end{proof}

\begin{example} For $\mathcal{T} = Rep(GL(r))$ let $\mathcal{M}$ be the monoidal subcategory of polynomial representations. A representation is polynomial if it a finite direct sum of irreducible representations $L(\lambda)$ of $GL(r)$ 
with $\lambda_r \!\geq\! 0$ for $\lambda=(\lambda_1,...,\lambda_r)$. Notice, the determinant representation $\mathcal{L} = L(1,\ldots,1)$ 
is polynomial, and for any algebraic finite dimensional representation $X$ of $GL(r)$ there exists $k$ such that ${\mathcal L}^k \otimes X$ is polynomial.
\end{example}





\section{Stabilization} \label{sec:stabilization}

It is a common tool (see e.g. \cite{Serganova-blocks}) in the study of block equivalences, to move an irreducible module via translation functors to a \emph{stable module}, where an irreducible module (or weight) in a block is called stable if all the atypical hooks $\vee$ are to the left of all core symbols. Our aim is to control $L(\lambda) \otimes L(\mu)$ via $L(\lambda_{st}) \otimes L(\mu_{st})$ up to negligible summands, where $\lambda_{st}$ and $\mu_{st}$ are stable weights associated to $\lambda, \mu$. Since repeated tensoring with the standard representation $(V,st)$ would destroy this property, we use shifts via some {\lq\textit{almost invertible}}\rq\ object $\Pi$ 
to circumvent this.

\subsection{Translation functors}

By \cite{Brundan-Stroppel-4}, taking the tensor product with $V$ decomposes as \[ - \otimes V = \bigoplus_{i \in \mathbb{Z}} \mathcal{F}_i (-) \] for the endofunctors $\mathcal{F}_i$, $i \in \Z$, from \cite[2.13]{Brundan-Stroppel-4}. Similarly we have a decomposition \[ - \otimes V^{\vee} = \bigoplus_{i \in \mathbb{Z}} \mathcal{E}_i (-) \] for the endofunctors $\mathcal{E}_i$, $i \in \Z$. On irreducible modules the effect of $\mathcal{F}_i$ can be understood diagrammatically. We quote from \cite[Lemma 2.4]{Brundan-Stroppel-4}.

\begin{lem}\label{translation}
Let $\lambda \in X^+$ and $i \in \mathbb Z$.
For symbols $x,y \in \{\circ,\wedge,\vee,\times\}$
we write $\lambda_{xy}$ for the diagram obtained from $\lambda$
with the $i$th and $(i+1)$th vertices
relabeled by $x$ and $y$, respectively.
\begin{itemize}
\item[\rm(i)]
If $\lambda = \lambda_{{\vee}\times}$ then $\mathcal{E}_i L(\lambda) \cong L(\lambda_{\times {\vee}})$. If $\lambda = \lambda_{\times \vee}$ then $\mathcal{F}_i L(\lambda) \cong L(\lambda_{{\vee} \times})$.
\item[\rm(ii)]
If $\lambda = \lambda_{{\wedge}\times}$ then $\mathcal{E}_i L(\lambda) \cong L(\lambda_{\times {\wedge}})$. If $\lambda = \lambda_{\times \wedge}$ then $\mathcal{F}_i L(\lambda) \cong L(\lambda_{\wedge \times})$.
\item[\rm(iii)]
If $\lambda = \lambda_{{\vee}\circ}$ then $\mathcal{F}_i L(\lambda) \cong L(\lambda_{\circ {\vee}})$. If $\lambda = \lambda_{\circ \vee}$ then $\mathcal{E}_i L(\lambda) \cong L(\lambda_{{\vee} \circ})$.
\item[\rm(iv)]
If $\lambda = \lambda_{{\wedge}\circ}$ then $\mathcal{F}_i L(\lambda) \cong L(\lambda_{\circ {\wedge}})$. If $\lambda = \lambda_{\circ \wedge}$ then $\mathcal{E}_i L(\lambda) \cong L(\lambda_{\wedge \circ})$.
\item[\rm(v)]
If $\lambda = \lambda_{{\times}\circ}$ then:
$\mathcal{F}_i L(\lambda)$ has irreducible
socle and head both isomorphic to $L(\lambda_{\vee\wedge})$, and all other composition
factors are of the form $L(\mu)$ for $\mu \in\lambda$
such that
$\mu = \mu_{\vee \wedge}$,
$\mu = \mu_{\wedge \vee}$ or
$\mu = \mu_{\wedge \vee}$.
Likewise for $\lambda = \lambda_{\circ \times}$ and $\mathcal{E}_i L(\lambda)$.
\item[\rm(vi)]
If $\lambda = \lambda_{{\vee\wedge}}$ then $\mathcal{F}_i L(\lambda) \cong L(\lambda_{\circ{\times}})$.
\end{itemize}
\end{lem}

\subsection{Left shifts}

For $r=m-n$ and $V + st_{m\vert n}$ we introduce the irreducible module  \[ \Pi := \Lambda^{r}(V) \otimes Ber_{m\vert n}^{-1}.\] Concerning this, notice that \[ \Pi \cong L(0,\ldots,0,-1,\ldots,-1 \ | \ 1,\ldots,1) \] is maximal atypical of superdimension $\pm 1$. Indeed, $\Lambda^{r}(V)$ has superdimension 
1 by the usual formula for the categorical dimension of exterior products. Furthermore $\Pi$ is irreducible since $\Lambda^{r}(V)$ is a direct summand in $V^{\otimes \ell}$ for some $\ell$ \cite{BR} \cite{Sergeev}. Hence \[ \Pi \otimes \Pi^{\vee}  \cong \one \oplus \text{ negligible module }.\] So, $\Pi$ is almost invertible in the sense that its image  in the semisimplification becomes invertible. We later denote this image $B_{core}^{-1}$, as it behaves like $Ber_{n\vert n}^{-1}$ in the case $m\!=\! n$.

\begin{cor} For any indecomposable module $M$ of non-vanishing superdimension $\Pi \otimes M$ is indecomposable up to negligible summands. 
\end{cor}

\begin{proof} Since $\Pi$ is invertible in the semisimplification $\overline{\mathcal T}_{m\vert n}=  \Tmnss$, tensoring with $\Pi$ sends an irreducible module to an irreducible module in $\overline{\mathcal T}_{m\vert n}$.   
\end{proof}

\begin{lem} The unique indecomposable summand of non-vanishing superdimension in $\Pi \otimes L(\lambda)$ is in the same block as $L(\lambda)$. 
\end{lem}

\begin{proof} Consider the tensor functor  \[ DS^n: \Tmn \to sRep(GL(r)) \] which sends $L(\lambda)$ to a multiple of $L(\chi_{\lambda})$.  Since $\Pi$ is in the principal block, $DS^n(\Pi) \cong \one$. Hence $DS^n(\Pi \otimes L(\lambda)) \cong DS^n(L(\lambda))$. Since the image under $DS^n$ determines the block, the result follows.
\end{proof}

\begin{prop} \label{moving-lemma} Let $L(\lambda)$ be maximal atypical. Then $\Pi \otimes L(\lambda)$ is irreducible up to negligible summands. The weight diagram of this irreducible summand in $\Pi \otimes L(\lambda)$ is obtained from the one for $\lambda$ by keeping the position of the crosses fixed and shifting every atypical hook $\vee$ by one to the next free vertex.
\end{prop}

\begin{proof} By Lemma \ref{translation} $\mathcal{F}_i(L(\lambda))$ either decreases the atypicality (in which case $\mathcal{F}_i(L(\lambda))$ is negligible), is zero or irreducible maximal atypical. Since $\Lambda^{m-n}(V)$ is a direct summand in $V^{\otimes m-n}$, $\Pi \otimes L(\lambda)$ is a direct summand in $V^{\otimes m-n} \otimes L(\lambda)$. It follows that \[ \Pi \otimes L(\lambda) \cong \text{ irreducible } \oplus \text{ negligible.} \] Here we use that any tensor product $V\otimes \mathcal{F}_i(L(\lambda))$ for negligible $\mathcal{F}_i(L(\lambda))$ is again negligible. This irreducible summand can be written as \[ Ber_{m\vert n}^{-1} \otimes (\mathcal{F}_{i_1} \circ \mathcal{F}_{i_2} \circ \ldots \circ \mathcal{F}_{i_{m-n}} (L(\lambda)).\] Via the list of Lemma \ref{translation} we need to consider only the following two elementary changes $\mathcal{F}_{i_j}$ at the positions $(i_j,i_{j+1})$ \[ \times \vee \mapsto \vee \times,  \ \ \ \times \wedge \mapsto \wedge \times.\] Such a sequence $\mathcal{F}_{i_1} \circ \mathcal{F}_{i_2} \circ \ldots \circ \mathcal{F}_{i_{m-n}}$ must shift each of the $m-n$ crosses exactly once to the left. For this note that $Ber^{-1}$ shifts all $\times$'s and all $\vee$'s one to the left. If a sequence $\mathcal{F}_{i_1} \circ \mathcal{F}_{i_2} \circ \ldots \circ \mathcal{F}_{i_{m-n}}$ would move any $\times$ not exactly once, $\mathcal{F}_{i_1} \circ \mathcal{F}_{i_2} \circ \ldots \circ \mathcal{F}_{i_{m-n}} \otimes Ber_{m\vert n}^{-1}$ would not be block preserving. But any sequence of length $m-n$ that shifts each of the $m-n$ symbols $\times$'s once has the same effect on $L(\lambda)$. If for fixed $i$ the vertices are labeled $\times \vee$ or $\times \wedge$, the elementary change turns this into $\vee \times$ and $\wedge \times$, and $Ber_{m\vert n}^{-1}$ shifts these to $\vee \times$ or $\wedge \times$ at the vertices $(i-1,i)$. 
\end{proof}

\begin{remark} The proof shows that the irreducible summand in $\Pi \otimes L(\lambda)$ may occur many times as a composition factor, but all the others are summands of some negligible module.
\end{remark}

\begin{cor} Let $L(\lambda)$ be maximal atypical. Then $\Pi^N \otimes L(\lambda)$ is irreducible and stable for $N >> 0$ up to negligible objects.
\end{cor}




\section{Separation and ground states} \label{sec:ground-states}

\subsection{Covariant representations} The tensor powers $V^{\otimes r}$ are completely reducible. The irreducible representations in such a tensor powers are called covariant.  For  the Schur functors $S_{\lambda}$ applied to $V$, denoted $S_{\lambda}(V)$, the irreducible covariant representations are parametrized by $S_{\lambda}(V)$ for certain partitions $\lambda$, described in 
\cite{BR} \cite{Sergeev}.

\subsection{Permanence properties of ground states}

Let ${\mathcal B}$ be a maximal atypical block. Let $j$ then be the minimum of the subset of all $\times$ in ${\mathcal B}$ or $j=1$ in case there are no $\times$'. The \emph{ground state} \cite{W} \cite{HW-tensor} of the block is the weight where the $\vee$'s are at the positions $j-1$, $j-2$, $\ldots$, $j-n$. Similarly we define higher ground states for $N=1,2,\ldots$, so that $L(\lambda_N)$ for $r=m-n$ is defined by $$ L(\lambda_1,\ldots,\lambda_{r}, \lambda_{r}\! -\! N, \ldots, \lambda_{r}\! -\! N \, | \! - \lambda_{r}\! +\! N, \ldots, - \lambda_{r}\! +\! N) \ . $$

\begin{lem} \label{lemma-ground} $DS$ maps ground states to ground states. The tensor product of two maximal atypical ground states decomposes into a direct sum of maximal atypical groundstates plus a direct sum of negligible objects.
\end{lem}

\begin{proof} The first claim follows immediately from Theorem \ref{mainthm}. For
$\mu_i = \lambda_i - \lambda_{r}$, twisting the ground state  gives \[ L(\lambda_0) \otimes Ber_{m\vert n}^{- \lambda_{r}} \cong L(\mu_1,\ldots,\mu_{r},0,\ldots,0 \ | \ 0,\ldots, 0), \] a covariant representation attached to the Schur functor $S_{\mu}$. Furthermore, by Lemma \ref{moving-lemma} (see also \cite[Lemma 3]{W}) \[ L(\lambda_N) \otimes \Pi \cong L(\lambda_{N+1}) \oplus \text{ negligible }\] for $N = 0,1,2,\ldots$. Therefore \begin{align*} L(\lambda_N) \otimes L(\widetilde{\lambda}_M) & \cong (L(\lambda_0) \otimes \Pi^N) \otimes (L(\widetilde{\lambda}_0 \otimes \Pi^M) \\ & \cong (S_{\mu}(V) \otimes Ber_{m\vert n}^{\lambda_{r}} \otimes \Pi^N) \otimes (S_{\widetilde{\mu}}(V) \otimes Ber_{m\vert n}^{\widetilde{\lambda}_{r}} \otimes \Pi^M) \\ & \cong \Pi^{N+M} \otimes (\ S_{\mu}(V) \otimes S_{\widetilde{\mu}}(V)\ ) \otimes Ber_{m\vert n}^{\lambda_{r} + \widetilde{\lambda}_{r}}.\end{align*} Notice, as stated in \cite[Lemma 2]{W} a covariant representation $S_{\mu}(V)$ is maximal atypical if and only if $\mu_{r+1} = 0$ and then \[ S_{\mu}(V) \cong L(\mu_1,\ldots,\mu_{r},0,\ldots,0 \, | \, 0,\ldots,0).\] Therefore all maximal atypical covariant representations are ground states. Since the tensor product of two covariant representations is covariant and both tensor products with Berezin powers and the representation $\Pi$ shift ground states to ground states up to negligible objects, the result follows.  
\end{proof}

\subsection{The classical group $H_{cl} \!=\! GL(r)$.} \label{section7.2}
For the semisimplification  ${\mathcal T}\!:\!=\! \overline{\mathcal{T}}_{m|n}^+ \!=\! \mathcal{T}_{m|n}^+/\mathcal{N}$ 
and ${\mathcal T}_{cl} = Rep(GL(r))$
we have tensor functors
$$    i: {\mathcal T}_{cl}  \hookrightarrow {\mathcal T}  \quad 
\mbox{and} \quad     \omega: {\mathcal T}  \longrightarrow {\mathcal T}_{cl}  $$
such that $\omega\circ i$ is an auto-equivalence of ${\mathcal T}_{cl}$. Besides $\omega$, already defined preceeding Corollary \ref{cor:fibre},
$i$ is induced from the full embedding of the tensor subcategory ${\mathcal T}_{cl}$ of ${\mathcal T}$
generated by the standard representation $V$ of $GL(m\vert n)$.
Needless to say that ${\mathcal T}_{cl}$ contains the image of $\det_{cl} = \Lambda^{r}(V)$,
an object of superdimension $1$. Hence $\det_{cl}$ is an invertible object of  ${\mathcal T}_{cl}$.
To show that
$\omega\circ i$ induces an auto-equivalence of ${\mathcal T}_{cl}$ requires
to identify the tensor categories $\mathcal{T}_{m-n|0}^+ = Rep(GL(r))$ and ${\mathcal T}_{cl}$,
which is clear since $\omega(\overline{V})=st_r$ holds for $r\geq 0$ and $st_r$ generates
the Tannaka category of $k$-linear representations of the classical group $GL(r)$.
Therefore, by Tannakian duality, $\omega\circ i : {\mathcal T}_{cl} \longrightarrow Rep(GL(r))$ induces a nontrivial algebraic 
group homomorphism $GL(r) \longrightarrow G({\mathcal T}_{cl})$.
To show that this morphism is an isomorphism of Tannaka groups, amounts to show the next 

\begin{lem} The category ${\mathcal T}_{cl}$ is equivalent to the tensor category
$Rep(GL(r))$.
\end{lem}

This Lemma is proven in \cite{Heidersdorf-mixed-tensors}, but we give another self-contained proof here.

\begin{proof} By \cite{BR}\cite{Sergeev} ${\mathcal T}_{cl}$ is a semisimple 
tensor category, and $\omega$ defines a fibre functor. Its Tannaka group $H_{cl}$ acts faithfully on
$\omega(\overline{V})$ for the generator $V=st_{m\vert n}$ with $\omega(\overline{V})\cong k^r$. So
$H_{cl}$ is a closed subgroup of $GL(r)$. Since $\omega\circ i$ induces a nontrivial algebraic group homomorphism
$GL(r) \to H_{cl}$ and since any nontrivial algebraic group homomorphism $GL(r)\to GL(r)$ is either an isomorphism
or $\det \cdot id_{k^r}$, it suffices to exclude the latter possibility. Since $\omega\circ i(st_{m\vert n})= st_r$ is irreducible and nontrivial for $r>1$,
the latter possibility only occurs for  $r=1$, where our assertion is evident anyway.  
This shows that $\omega\circ i$ induces a canonical isomorphism $GL(r) \cong  H_{cl}$.
\end{proof}

This being said, it is now clear that $i$ and $\omega$ 
are tensor functors between ${\mathcal T}_{cl}$ and the semisimplification  ${\mathcal T} = \overline{\mathcal{T}}_{m|n}^+$ such that
$\omega \circ  i$ is an autoequivalence of ${\mathcal T}_{cl}$. Because of corollary \ref{cor:fibre} the functor $i$ induces a surjection $H\to H_{cl}$ that is split
by the homomorphism $s: H_{cl} \to H$, induced from $\omega$.

\subsection{The classical and principal part} After tensoring with $\Pi$ sufficiently often, we may assume that any maximal atypical irreducible representation is stable. Furthermore, the following lemma follows immediately from the definitions.

\begin{lem} \label{unimportant} For $r = m-n$ an irreducible representation $L(\lambda)$ is stable if and only if it is of the form \[L(\lambda) = L(\lambda_1,\ldots,\lambda_r,\mu_1, \ldots, \mu_n \, | \, - \mu_n, \ldots,-\mu_1) \ ,\]
also written as $L(\lambda) = L(\lambda_1,\ldots,\lambda_r, \mu \, | \, - \mu)$ for $\mu \in \mathbb{Z}^n$.
\end{lem}

Recall that Serganova has shown (\cite[Theorem 3.6]{Serganova-blocks}; see also \cite[Proposition 4.7.2]{BKN-2}) that as abelian category any maximal atypical block $\mathcal B$ of $\Tmn$ is equivalent to the principal block 
${\mathcal B}_{\mathbf 1}$ of ${\mathcal T}_{n\vert n}$. 
On the full subcategory of stable modules in the given block this equivalence is given by $\eta_{\lambda}(X)= \{x\in X \vert h(X)=\lambda(h) \!\cdot\! X \ \forall h \in \mathfrak{h}'  \}$
for a suitable additive homomorphisms $\lambda$ of the Cartan Lie algebra  $\mathfrak{h}'$ of the diagonal torus of $GL(r)$. For objects in the
principal block the relevant choice is $\lambda=0$.
For the stable representation $L(\lambda)$ of Lemma \ref{unimportant} the image
under the functor $\eta_{\lambda}$ is the irreducible representation $[\mu_1,\ldots,\mu_n]$ in the principal block of $GL(n|n)$.

\begin{definition}\label{principal reps}  A stable representation is negatively stable if $\mu_1 \leq 0$.
Given a negatively stable irreducible representation $L(\lambda)$, we separate it into two parts: The classical part and the principal part. For given $L(\lambda)$ the principal part  $\lambda_{pr}$ is defined by \[ \lambda_{pr} = (0,\ldots,0,\mu_1,\ldots,\mu_n \, | \, - \mu_n,\ldots,-\mu_1) \ .\]
Furthermore, the classical part  $\lambda_{cl}$ is defined by  \[ \lambda_{cl} = (\lambda_1,\ldots,\lambda_r,0,\ldots,0\, |\, 0,\ldots,0)\ .\] 
\end{definition}

\subsection{Splicing} We have two distinguished invertible objects  in the semisimplification $\Tmn$, the images of $\Pi$ and $Ber_{m\vert n}$. Tensoring with the representation $\Pi$ shifts the principal part and leaves the classical core part, encoded by $L({\chi_{\lambda}})$, unchanged up to negligible objects. Tensor products with the representation $\Pi \otimes Ber_{m\vert n} = \Lambda^{r}(V)$ preserve the principal part, but shift the classical part  by a tensor product with the determinant $det$. Since $DS$ commutes with Schur functors and induces $d$,
for the enriched fibre functor $\omega = d^n$ indeed $\omega(\Pi \otimes Ber_{m\vert n}) = det$ is the determinant of $GL(r)$.

\begin{thm} \label{splicing} (The splicing theorem) Any maximal atypical negatively stable irreducible $L(\lambda)
= L(\lambda_1,\ldots,\lambda_r,\mu_1, \ldots, \mu_n |  - \mu_n, \ldots,-\mu_1) $ with $\lambda_r \geq 0$ is the tensor product of its classical part with its principal part \[  L(\lambda_{cl})  \otimes L(\lambda_{pr})\cong 
 L(\lambda) \ \ \text{mod } \mathcal{N}.\]
\end{thm}

\begin{proof} In the semisimplification the standard representation generates a tensor subcategory equivalent to $Rep(H_{cl})$ for $H_{cl}=GL(r)$; see Section \ref{section7.2}. We first claim that \[L(\lambda_{cl}) \otimes L(\lambda_{pr}) = \bigoplus L(\nu_i) \oplus \text{ negligible}.\] Since $L(\lambda_{cl})$ is a direct summand in some $V^{\otimes d}$, it suffices to show \[  V^{\otimes d} \otimes  L(\lambda_{pr}) = \bigoplus L(\nu_i) \oplus \text{ negligible}.\] The same argument as in Lemma \ref{moving-lemma}, using an analysis of translation functors, now proves the claim. This implies \[L(\lambda_{cl}) \otimes L(\lambda_{pr}) = L(\lambda) \oplus \bigoplus_{\nu_i \neq \lambda} L(\nu_i) \oplus \text{ negligible}\] for the highest weight \[\lambda_{cl}+\lambda_{pr} = \lambda = (\lambda_1,\ldots,\lambda_r,\mu_1, \ldots, \mu_n \ | \ - \mu_n, \ldots,-\mu_1).\] For the superdimensions notice $\sdim L(\lambda) = m(\lambda) \sdim L(\chi_{\lambda})$. Furthermore $\omega(L(\lambda_{cl})) = L(\chi_{\lambda})$ and $\omega (L(\lambda_{pr})) \cong m(\lambda) \one_{GL(r)}$, and the isotypic multiplicities agree since they depends only on the nesting structure in the cup diagram of $\lambda$. Hence the superdimensions of $L(\lambda_{cl}) \otimes L(\lambda_{pr})$ and $L(\lambda)$ agree. Thus the claim follows from   theorem \ref{thm:tannaka}.
\end{proof}

Using suitable twists with the two determinants objects $\Pi$ and $\Lambda^{m-n}(st)$ allows to extend the splicing to an arbitrary $L(\lambda)$.

\begin{cor} \label{cor-splicing} Any maximal atypical irreducible representation $L$ defines representations $L_{cl}=L(\lambda_{cl})$ and $L_{pr}=L(\lambda_{pr})$ such that \[ L \cong L(\lambda_{cl}) \otimes L(\lambda_{pr}) \ \text{ modulo } \mathcal{N}\] holds for $L(\lambda_{pr}) \cong \Pi^N \otimes L(0,\ldots,0,\mu_1,\ldots,\mu_n  | - \mu_n,\ldots,- \mu_1)$ for suitable $N$ and suitable weight $0 \geq \mu_1\geq \mu_2 \ldots \geq \mu_n$.
\end{cor}

\subsection{Compatibilities of Tannaka duality}\label{duality}

By \cite{GS} the block equivalence $\eta_{\lambda}$ to the principal block can be diagrammatically described as follows: Remove all core symbols $\times$ and $\circ$ and fill empty positions by filling with symbols from the right.

\medskip
The module $\Pi$ shifts the symbols $\vee$ exactly like the inverse Berezin representation in the principal block of $GL(n|n)$. As we refer to the equivalent representation in the principal block ${\mathcal B}_{\mathbf 1}$ as the principal core, the property $\eta_0(\Pi) \cong Ber_{n|n}^{-1}$, we also use the notation \[ B_{core}^{-1} := \Pi \text{ and } B_{core} := (B_{core}^{-1})^{\vee}.\] We use the shortcut notation $B_{core}^r = B_{core}^{\otimes r}$ for $r \in \mathbb{Z}$ inspired
by the later example \ref{BerCompatibility}.

In \cite[Section 2.3]{HW-tannaka} we defined an irreducible representation $L(\lambda) \in \mathcal{T}_{n|n}$ to be of type (SD) if there exists an isomorphism $L(\lambda) \cong L(\lambda)^{\vee} \otimes Ber_{n|n}^{k}$ for some $k \in \mathbb{Z}$.

\begin{definition} \label{new-duality} Let $L(\lambda)$ be a stable module in the principal block $\mathcal{B}_{\one} \subset \mathcal{T}_{m|n}$. We say that $L(\lambda)$ is of type (SD) if there exists $k \in \mathbb{Z}$ such that \[ L(\lambda)^{\vee} \oplus \text{ negligible summands } \cong L(\lambda) \otimes B_{core}^k.\]
\end{definition}

\begin{lem} Suppose $L(\lambda)$ is a stable module in the principal block. Then $L(\lambda)$ is (SD) if and only if $\eta_0(L(\lambda))$ is (SD) and $L(\lambda)^{\vee} \oplus \text{ negligible summands }  \cong L(\lambda) \otimes B_{core}^k$ if and only if $\eta_0(L(\lambda))^{\vee} \cong \eta_0(L(\lambda)) \otimes Ber_{n|n}^k$.
\end{lem}

\begin{proof} By Proposition \ref{irreducible-dual} the descriptions of the duals is the same except that we have the reflection $s \mapsto 1 - r -s$ for $r=m-n >0$ and $s \mapsto 1-s$ for $m=n$. This is due to the fact that $r$ entries $\vee$ have to be reflected over the $r$ entries $\times$. The position of the crosses stays fixed; and $B_{core}^{-1}$ moves the $\vee$'s exactly like $Ber^{-1}_{n|n}$.
\end{proof}

\subsection{Product decomposition} Let $H_{pr}$ be the Tannaka group
of the tannakian subcategory of ${\mathcal T}_{pr}$ of  ${\mathcal T}=\Tmn^+/\mathcal{N}$, generated
by the images of the irreducible representations of the principal block ${\mathcal B}_{\mathbf 1}$. For the Tannaka group $H=H_{m\vert n}$ of $\mathcal T$, the full embedding ${\mathcal T}_{pr} \subset \mathcal{T}$
corresponds to a surjection $pr: H\twoheadrightarrow H_{pr}$.
Generators of the tannakian category ${\mathcal T}_{pr}$
are thus the representations  $L(\lambda_{pr})$ considered
in Section \ref{principal reps}.

\medskip For the surjection $p: H\twoheadrightarrow H_{cl}=GL(r)$, attached to the tannakian 
subcategory ${\mathcal T}_{cl}$ considered 
in Section \ref{section7.2},
consider the homomorphism $p\times pr: H \to H_{cl}\times H_{pr}$.
By Corollary \ref{cor-splicing}, the tannakian category
$\mathcal T$ is generated by the tannakian subcategories ${\mathcal T}_{cl}$
and ${\mathcal T}_{pr}$, i.e. a set of generators of ${\mathcal T}_{cl}$ and ${\mathcal T}_{pr}$
generates $\mathcal T$. This immediately implies that the above map $ H \to  H_{cl}\times H_{pr} $
is an injection $$ H \hookrightarrow  H_{cl}\times H_{pr} \ .$$

\begin{lem} \label{product-decomposition} For $H_{cl}=GL(r)$ and the pro-reductive group $H_{pr}$ the following holds:  $$ p\times pr: H_{m|n}  \cong H_{cl} \times H_{pr} \ .$$ Furthermore $s: H_{cl} \to H_{m\vert n}$ is the inclusion onto the first
 factor.
\end{lem}

\begin{proof} The image $\omega({\mathcal T}_{pr})$ of  ${\mathcal T}_{pr}$
under the functor $\omega: {\mathcal T} \to {\mathcal T}_{cl}$ is the subcategory of trivial representations in $Rep(GL(r))$. This easily follows from Theorem \ref{mainthm}, more precisely its consequence Lemma \ref{lem:commutes}.
By Tannaka duality 
this means that the composition $pr\!\circ\! s$ of the projection $pr: H\to H_{pr}$ and the homomorphism $s: H_{cl}\to
H$ of Section \ref{section7.2} is trivial.  Hence $s(H_{cl})\subset ker(pr)$ holds, 
and considered in $ H_{cl}\times H_{pr}$ instead of $H$, we obtain $s(h_{cl})= (h_{cl},1)$.
Because $pr$ is surjective, for $h_{pr}$ in $H_{pr}$  in $H_{cl}\times H_{pr}$ there exists $h=(h_{cl},h_{pr})$ 
contained in $H$.
Hence $\tilde h= s(h_{cl})^{-1}h$ is in $H$, 
and it is of the form $\tilde h=(1,h_{pr})$ in $H=H_{cl}\times H_{pr}$. From this we obtain $H=H_{cl}\times H_{pr}$.
\end{proof}


For indecomposable objects $X \in\mathcal{T}_{m|n}^+$ their image $\overline{X}$ in $\mathcal T$ can therefore be written as an external tensor product \[ \overline{X} = \rho_{cl} \boxtimes \sigma_{H_{pr}} = \rho \boxtimes \sigma\] for an irreducible representation $\rho = \rho_{cl}$ of $GL(r)$ and an irreducible representation $\sigma = \sigma_{H_{pr}}$ of the Tannaka  group $H_{pr}$. 
A priori it is unclear how irreducible representations $\sigma_{H_{pr}}$ of $H_{pr}$ are represented by 
indecomposable objects of the 
category ${\mathcal T}=\mathcal{T}_{m|n}^+$.

Lemma \ref{product-decomposition} implies that $\mathcal{T}_{m|n}^+/\mathcal{N} \cong Rep(H_{cl} \times H_{pr})$ contains the full tannakian subcategories $Rep(\one \times H_{pr}) \cong Rep(H_{pr})$ and $Rep(H_{cl} \times \one) \cong Rep(H_{cl})$. In particular, these subcategories are closed under tensor products and duality.

\section{The groups $H_{m|n}$} \label{sec:groups}

\subsection{The $GL(m|1)$-case} This case is special since blocks have tame representation type. In this case, in  \cite{H-ss} \cite{Heidersdorf-mixed-tensors} it was shown for $\overline{\mathcal T}_{m\vert 1}=\mathcal{T}_{m|1}^+/\mathcal{N}$ that \[ \overline{\mathcal T}_{m\vert 1} \cong Rep(GL(m\! -\! 1)\times GL(1)).\] In general, the semisimplification of the full karoubian subcategory of $\mathcal{T}_{m|n}^+$, generated by the standard representation and its dual, coincides with $Rep(GL(r))$ for $r=m-n$. 


\subsection{Decomposition of representations} We return now to the decomposition 
\[ H  \cong H_{cl} \times H_{pr}\] of the pro-reductive Tannaka group $H\!=\! H_{m|n}$. Recall that for indecomposable $X$ its image in the semisimplification can be written as $\overline{X} \cong \rho \boxtimes \sigma$ for an irreducible representation $\rho = \rho_{cl}$ of $H_{cl}$ and an irreducible representation $\sigma = \sigma_{H_{pr}}$.
For the images in ${\mathcal T}\!:\!=\! \overline{\mathcal{T}}_{m|n}^+$ of irreducible objects $X\in {\mathcal{T}_{m|n}^+}$ 
the external tensor product  \[ \overline{X} = \rho \boxtimes \sigma = (\one \boxtimes \sigma) \otimes (\rho \boxtimes \one)\] is explicitly realized by irreducible representations of $GL(m|n)$ that are given by the Splicing Theorem \ref{splicing} and Corollary \ref{cor-splicing}. If $X_{\lambda}$ satisfies the conditions of Theorem \ref{splicing}, the representation $\sigma$ is represented by
the irreducible object $$L(\lambda_{pr})=L(0,....,0,\mu_1,...,\mu_n\vert -\mu_n,...,-\mu_1)\ ,$$ whereas
$\rho$ is realized by the image in ${\mathcal T}_{cl} \subset {\mathcal T}$ of the irreducible tensor $$L(\lambda_1,...,\lambda_r,0,...,0\, \vert\, 0,...,0)\ .$$ 
To see this, it suffices to verify that $\omega(L(\lambda_{cl}))$ is isomorphic to the irreducible representation
of $GL(r)$ of highest weight $\lambda=(\lambda_1,...,\lambda_r)$ and $\omega(L(\lambda_{pr}))$ is isomorphic the the 
one-dimensional trivial representation
of $GL(r)$. Obviously, both statements follow from the explicit description in \ref{mainthm}
of the functor $DS^n$ that induces $\omega=d^n$.

\begin{example}\label{BerCompatibility} Recall $\Pi= Ber_{m\vert n}^{-1} \otimes  \Lambda^{r}(V)$.
We have the following images in the semisimplification $\overline{\mathcal{T}}_{m|n} = Rep(GL(r)\times H_{pr})$:
$$ det_{cl} = \Lambda^{r}(V) = L(1,\ldots,1,0,\ldots,0 \, | \, 0,\ldots,0)  \ \ \mapsto\ \  L(1,\ldots,1) \boxtimes \one = \det \boxtimes \one$$ 
Therefore the image of $det_{cl}$ in $\overline{\mathcal{T}}_{m|n}$ is $det\boxtimes \one$. So any object in
 ${\mathcal T}_{cl}$ is represented by $(det_{cl}^{\vee})^{\otimes \ell} 
 \otimes L(\lambda_1,...,\lambda_r,0,...,0\, \vert\, 0,...,0)$ for some $\ell$ in $\N$ and $\lambda_1\geq ... \geq \lambda_r\geq 0$, up to negligible summands.
Under the projection to the semisimplification 
$$ Ber_{m|n} = L(1,\ldots,1, 1, \ldots 1  \,| \! -1,\ldots,-1) \ \ \mapsto \ L(1,\ldots,1) \boxtimes \overline{Ber}_{n|n} = \det \boxtimes \overline{Ber}_{n|n} $$
furthermore implies that the image of $\Pi$ is $\one \otimes \overline{Ber}_{n\vert n}^{-1}$. Therefore, for all $\ell$ in $\mathbb Z$ we have 
$$ \Pi^{\otimes \ell} \ \ \mapsto \ \ \one \boxtimes \overline{Ber}_{n|n}^{\otimes \ell}. $$
\end{example}

\subsection{Reduction to the principal block} \label{reductions}

In order to determine the unknown group $H_{pr}$, it is enough to determine the tensor subcategory generated by the image of the principal block ${\mathcal B}_{\mathbf 1}$ in $\overline{\mathcal T}_{m\vert n}=\mathcal{T}_{m|n}^+/\mathcal{N}$.

\smallskip
To see this, consider now the tensor product $X_{\lambda} \otimes X_{\mu}$ of two irreducible maximal atypical representations. We would like to compute the decomposition into indecomposable summands up to superdimension $0$. In $\overline{\mathcal T}_{m\vert n}$ we write for the images \[ \overline{X}_{\lambda} = \rho^{\lambda} \boxtimes \sigma^{\lambda}, \ \ \overline{X}_{\mu} = \rho^{\mu} \boxtimes \sigma^{\mu}\] and therefore \begin{align*} \overline{X}_{\lambda} \otimes \overline{X}_{\mu} & \cong ( \rho^{\lambda} \boxtimes \sigma^{\lambda}) \otimes ( \rho^{\mu} \boxtimes \sigma^{\mu}) \\ & \cong (\one \boxtimes \sigma^{\lambda}) \otimes (\rho^{\lambda} \boxtimes \one) \otimes (\one \boxtimes \sigma^{\mu}) \otimes (\rho^{\mu} \boxtimes \one).\end{align*}
So the tensor product decomposition boils down to compute tensor products of the following two situations: First \[ (\rho^{\lambda} \boxtimes \one) \otimes (\rho^{\mu} \boxtimes \one) \cong ( \rho^{\lambda} \otimes \rho^{\mu} \boxtimes \one)\ ,\] 
 easily calculated by the Littlewood-Richardson rule.
Second, the tensor products  \[ (\one \boxtimes \sigma^{\lambda}) \otimes (\one \boxtimes \sigma^{\mu}) \cong (\one \boxtimes \sigma^{\lambda} \otimes\sigma^{\mu}).\] To compute a tensor product $L\otimes L'$ we may therefore assume that the image of $L,L'$ has the form $\one \otimes \sigma^L$. Note, by the Splicing Theorem and its corollary these may be  represented (after shifts by $\Pi$) by stable modules in the principal block.

\smallskip
Since maximal atypical blocks ${\mathcal B}$ correspond 1:1 to irreducible representations of $GL(r)$ via ${\mathcal B} \leftrightarrow L(\chi_{\lambda})$, the image of the principal block  ${\mathcal B}_{\mathbf 1}$ of $GL(m|n)$ in the semisimplification is equivalent to $\{\one\} \times Rep(H_{pr}) \cong Rep(H_{pr})$. Therefore the determination of $H_{m|n}$ boils down to determine this image.
We claim that $$H_{pr} \cong H_{n|n} \ ,$$ or, in other words, that the image of the principal block ${\mathcal B}_{\mathbf 1}$  in $\mathcal{T}_{m|n}^+/\mathcal{N}$ generates a tensor category which is equivalent to the one generated by the principal block of $GL(n|n)$ in $\mathcal{T}_{n|n}^+/\mathcal{N}$. We show this first for the images of the \emph{negative halfs} of the principal blocks in the next section.

\subsection{The restriction map.} 

The obvious idea to relate $H_{m|n}$ and $H_{n|n}$ is to consider the  embedding $H_{cl} \times GL(n|n)$ where $H_{cl} = GL(r)$ is embedded
in the left upper $(r\times r)$-corner and  $ GL(n\vert n)$ in the right lower $(2n\times 2n)$ corner
of matrices in $GL(m\vert n)$. Consider the functor $res$ \[ res: \mathcal{T}_{m|n} \to Rep(GL(r)) \otimes \mathcal{T}_{n|n}\,  \] given by the restriction of representations of
$GL(m\vert n)$ to the subgroup $GL(r) \times GL(n\vert n)$ with respect to this embedding. While $res$ is an exact tensor functor, it is not obvious that it restricts to a tensor functor between the $\mathcal{T}^+$-subcategories because of the following two problems:

\begin{itemize}
\item The functor $res$ might not send irreducible modules to semisimple modules. 
\item We don't know whether $res$ preserves negligible modules, and so we don't know that all objects in the image have superdimension $\geq 0$.
\end{itemize}

To cope with this, we enlarge our categories to the categories $\mathcal{T}^{ev}$ defined below. By the commutative diagrams $(*)$ of Lemma \ref{lem:commutes}, we see these categories are obviously preserved by $res$.

\medskip
For $X$ in ${\mathcal T}_{m\vert n}$ the object $DS^n(X)$ is the direct sum of an even and an odd finite dimensional
representation of the classical group $GL(r)$ over $k$, meaning that $DS^n(X) = DS^n(X)_{\overline 0}\oplus
DS^n(X)_{\overline 1}$  in
$sRep_k(GL(r))$. We say $X$ is even
if the odd summand $DS^n(X)_{\overline 1}$ is zero.  
Let ${\mathcal{T}_{m|n}^{ev}}$ be the full subcategory of even objects in $X\in {\mathcal T}_{m\vert n}$, as considered in \cite[Section 14.1]{HW-tannaka}. In \cite[Section 24]{HW-tensor} it is shown that simple objects in $\mathcal T_{m\vert n}$ are always even or odd, hence \[ \mathcal{T}^+_{m\vert n} \subset \mathcal{T}^{ev}_{m\vert n}.\]
The inclusion $\mathcal{T}^{+}_{m\vert n} \!\subset\! \mathcal{T}^{ev}_{m\vert n}$  is strict. Already for $GL(1|1)$, $\mathcal{T}^{ev}_{1\vert 1}$ contains zigzag modules of length $2m+1$ for $m \in \mathbb{N}$ \cite{H-ss} which are not contained in $\mathcal{T}_{1|1}^+$, and they have nonvanishing superdimension.

\medskip
Obviously ${\mathcal{T}_{m|n}^{ev}}$ is an exact subcategory of  ${\mathcal{T}_{m|n}}$, 
closed under Tannaka duality, tensor products, retracts and extensions. 

\begin{lem} \label{lem:commutes}
The tensor functor
$$ DS^n:  {\mathcal{T}_{m|n}^{ev}} \to {\mathcal{T}_{r|0}^{ev}} = Rep(H_{cl}) $$
is an exact functor since $H_{cl}$ is reductive, and the diagram $(*)$
$$  \xymatrix{   {\mathcal{T}_{m|n}^{ev}}  \ar[rr]^-{res}\ar[d]_{DS^n} & &  Rep(H_{cl}) \otimes {\mathcal{T}_{n|n}^{ev}}  
\ar[d]^{id\otimes DS^n} \cr
                           Rep(H_{cl}) \ar@{=}[rr] & &   Rep(H_{cl}) } $$
is commutative. There is a similar commutative diagram  
$$  \xymatrix{   {\mathcal{T}_{m|n}}  \ar[rr]^-{res}\ar[d]_{DS^n} & &  Rep(H_{cl}) \otimes {\mathcal{T}_{n|n}}  
\ar[d]^{id\otimes DS^n} \cr
                           sRep(H_{cl}) \ar@{=}[rr] & &   sRep(H_{cl}) } $$ where $sRep(H_{cl})$ denotes the super representations of $H_{cl}$.           
\end{lem}

\begin{proof} It is enough to show that the element $x$, used in the definition of $DS$, commutes with the embedded $GL(r) \times GL(n\!-\! 1|n\!-\! 1)$. This is an easy calculation and shows the commutativity for $DS$.  Applying this repeatedly gives the result for $DS^n$. 
\end{proof}                           
     
\medskip
Let $\rho$ denote an irreducible representation of $GL(r)$
with highest weight $(\lambda_1,...,\lambda_r)$, and let $\eta(X)$
denote $Hom_{GL(r)}(\rho,res(X))$. Since this is an invariant functor under a reductive group, $\eta$ defines
an exact $k$-linear functor $\eta: {\mathcal T}_{m|n} \to {\mathcal T}_{n|n}$.

\medskip
From now on assume that $\rho$ is the trivial representation.

\begin{lem} \label{eta-negligible} The exact $k$-linear functor $\eta:\mathcal{T}_{m|n}^{ev} \to \mathcal{T}_{n|n}^{ev}$ sends negligible objects to negligible objects.
\end{lem}

\begin{proof} This follows from the commutative diagram (*). The restriction functor $res$ preserves superdimensions and since all objects in $\mathcal{T}_{m|n}^{ev}$ have superdimension $\geq 0$ this implies that any direct summand in the restriction $res(X)$ of an indecomposable negligible object $X$ is negligible, and therefore also in its retract $\eta(X)$. 
\end{proof}

Notice however that $\eta: {\mathcal{T}}_{m|n} \to {\mathcal{T}}_{n|n} $ is not a tensor functor. 

\smallskip
From Corollary \ref{cor:fibre} recall that the forgetful fibre functor $\omega_r\!:\! Rep(GL(r))\!\to\! vec_k$ and the 
functor $\omega: \overline{\mathcal{T}}_{m|n} \to Rep(GL(r))$ define the fibre functor $\omega_{m\vert n} =  \omega_r \circ \omega :\overline{\mathcal{T}}_{m|n} \to vec_k$. For $G=GL(r)$ let  $(.)^{G}\!:\! Rep(GL(r))\!\to\! vec_k$
denote  the functor
of $G$-invariants. Lemma \ref{product-decomposition} gives the decomposition
$\overline{\mathcal{T}}_{m|n} = Rep(GL(r)\times H_{pr}) = Rep(GL(r))\otimes  Rep(H_{pr})$.

\begin{lem} \label{lem:fibre} The following holds:
\begin{enumerate}
\item The restriction of $\omega$  to the subcategory $Rep(H_{pr})$ of 
$\overline{\mathcal{T}}_{m|n}$ has values in the tensor subcategory $\langle \mathbf 1\rangle$ of $Rep(GL(r))$ generated by the trivial representation $\mathbf 1$.
\item Restricted  to the subcategory $Rep(H_{pr})$, the fibre functor $\omega_{m\vert n}\!:\! \overline{\mathcal{T}}_{m|n}\! \to\! vec_k$ coincides with the functor  $(.)^G \otimes \omega$.
\item The composition $DS^n\circ \eta$ of the functors $\eta\!:\! {\mathcal{T}}_{m|n}\! \to\! {\mathcal{T}}_{n|n}  $ and $DS^n\!:\! {\mathcal{T}}_{n|n}\! \to\! vec_k$ and the composition $(.)^G\circ DS^n$ of the functors $DS^n\!:\! {\mathcal{T}}_{m|n} \!\to\! Rep(GL(r))$ and 
$(.)^G\!:\! Rep(GL(r))\!\to\! vec_k$ coincide.
\end{enumerate}
\end{lem}

\begin{proof}
(1)  immediately follows from the construction of the product decomposition 
of  Lemma \ref{product-decomposition}. Assertion (2) follows from (1) since on $\langle \mathbf 1 \rangle \subset Rep(GL(r))$ the forgetful functor
$\omega_r$ identifies $ \langle \mathbf 1 \rangle$ with $vec_k$ and coincides with
the functor $(.)^{G}$. Concerning assertion (3), notice by Section \ref{ds-functor} we defined $DS$  
as the cohomology with respect to an element $x$ in the Lie superalgebra $\mathfrak{gl}(m\vert n)$ that commutes
with the the adjoint action of our chosen subgroup $GL(r)$ of $GL(m\vert n)$.
\end{proof}

\subsection{Schur-Weyl duality revisited} \label{sec:sw-duality}

The restriction of our fixed embedding $$GL(r) \times GL(n\vert n) \hookrightarrow GL(m\vert n)\ $$
to the first factor $GL(r)$ induces  an embedding of the group $GL(r)$ into $GL(m\vert n)$. Restriction with respect to this embedding then 
defines a restriction functor $res_{GL(r)}: Rep(GL(m\vert n)) \to Rep(GL(n\vert n)$. If $0\geq \lambda_1 \geq \lambda_2 \geq ... \geq \lambda_r$ holds for an irreducible representation $L(\lambda)$ 
of $GL(r)$ and highest weight $\lambda=(\lambda_1,...,\lambda_r)$, we will write $\lambda\leq 0$. 

Let ${\mathcal T}_{m\vert n}^{\leq 0}$ denote the intersection of $\mathcal{T}_{m|n}^{ev}$ with the full abelian category
of all objects in ${\mathcal T}_{m\vert n}^{ev}$ whose irreducible Jordan-H\"older
constituents $L(\lambda)$ have highest weights $\lambda$ with $\lambda_1\leq 0$.

\begin{lem} \label{SW} For $X$ in ${\mathcal T}_{m\vert n}^{\leq 0}$ the following
properties are equivalent: 
\begin{enumerate}
\item $X\in {\mathcal T}_{m\vert n}^{\leq 0}$.
\item As a representation of $GL(r)$ the representation $res_{GL(r)}(X)$
decomposes into a direct sum of irreducible representation $\bigoplus_{\lambda\leq 0} L(\lambda)$.
\item The dual representation $res_{GL(r)}(X)^{\vee}$ is a polynomial representation of $GL(r)$.
\item All irreducible constituents of $res_{GL(r)}(X)^{\vee}$ are tensors, i.e.
occur in  $\bigoplus_{k\geq 0} st_r^{\otimes k}$. 
\end{enumerate}
\end{lem}

\begin{proof} To show the equivalence of (1) and (2) we can assume $X$
to be irreducible. For irreducible $X\!=\! L(\lambda,\mu\vert \mu')$ all weights
of $X$ are less or equal to the highest weight $(\lambda,\mu\vert \mu')$.
This immediately carries over to the restriction $res_{GL(r)}(X)$, whose
$GL(r)$ weights are therefore less or equal to the $GL(r)$-weight $\lambda$ of the
diagonal subtorus $T$ of $GL(r)$. Therefore (1) implies (2).
Conversely (2) implies (1).

\medskip
The equivalence of (2) and (3) is immediate since $L(\lambda)^{\vee} \cong
L(\mu)$ for $\mu=(-\lambda_r,...,-\lambda_1)$ and an irreducible representation $L(\mu)$ of $GL(r)$ is polynomial
if and only if $\mu_r\geq 0$. Finally (3) and (4) are equivalent by the Schur-Weyl main theorem of invariant theory,
according to which in particular an irreducible representation $L(\lambda)$ is polynomial (i.e. $\lambda_r\geq 0$ holds)
if and only if $L(\lambda)$ is a constituent of the tensor algebra $\bigoplus_{k\geq 0} st_r^{\otimes k}$.
\end{proof}

In the following we identify $\mathbb G_m$
with the one dimensional subtorus of $GL(r)$ given by the elements $diag(id_{r-1},t)$ for $t$ in $k^*$.
For the block diagonal embedding $(g,t)\mapsto diag(g,t)$
of $GL(r\!-\! 1)\times\mathbb G_m$ into $GL(r)$ and representations $L$ of $GL(r)$, the space of invariants $L^{\mathbb G_m}\subseteq L$ under the torus $\mathbb{G}_m$ defines a representation of $GL(r\!-\! 1)$.

\begin{lem} \label{Schur-Weyl}
For an irreducible constituent $L=L(\lambda)$ of the tensor algebra $\bigoplus_{k\geq 0} st_r^{\otimes k}$
the following holds:
\begin{enumerate}
\item $L(\lambda)^{\mathbb G_m}\neq 0$ iff $\lambda_r=0$ holds for the weight $\lambda=(\lambda_1,...,\lambda_r)$.
\item If $L(\lambda)^{\mathbb G_m}\neq 0$, then $L(\lambda)^{\mathbb G_m}\cong L(\mu)$ holds as an irreducible  representation of $GL(r\!-\! 1)$ of the highest weight $\mu=(\lambda_1,...,\lambda_{r-1})$.
\end{enumerate}
\end{lem}

\begin{proof} Although this is well-known, for the convenience of the reader we recall the proof:
All weights of $st_r$, hence all weights $\mu=(\mu_1,...,\mu_r)$ of $\bigoplus_{k\geq 0} st_r^{\otimes k}$ and of $L$
satisfy  $\mu_r\geq 0$. If $\mu_r>0$ holds for all weights $\mu$ of $L$,
then  $\lambda_r>0$ holds. Conversely $\lambda_r>0$ 
implies that $\tilde L = L \otimes det^{-1}$ is polynomial and hence a tensor by the Schur-Weyl theorem,
and this implies $\mu_r>0$ for all weights of $L$.
This being said,
notice that $L(\lambda)^{\mathbb G_m} \neq 0$ holds if
and only there exist at least one weight $\mu$ of $L=L(\lambda)$ such that $\mu_r=0$ holds. 
But, as explained above, the latter is equivalent to the assertion $\lambda_r=0$.
This proves the first claim. The second claim follows from the explicit description
of $L(\lambda)$ as the image in $\bigoplus_{k\geq 0} st_r^{\otimes k}$ under the Schur projector
$S_\lambda$. If $\lambda_r=0$, the Schur projector $S_\lambda$ on
$\bigoplus_{k\geq 0} st_r^{\otimes k}$ induces the Schur projector $S_\mu$ on
$\bigoplus_{k\geq 0} (st_r^{\otimes k})^{\mathbb G_m} = \bigoplus_{k\geq 0} st_{r-1}^{\otimes k}$
 for $\mu=(\lambda_1,...,\lambda_{r-1})$ and $k=deg(\lambda)=deg(\mu)$,
noticing $(st_r^{\otimes k})^{\mathbb G_m} = (st_{r-1}^{\otimes k})$. This implies assertion (2).
\end{proof}

\begin{cor} \label{invariants} For the torus $T$ of diagonal matrices in $GL(r)$ and
for irreducible constituents $L$ of the tensor algebra $\bigoplus_{k\geq 0} st_r^{\otimes k}$
the following holds:
$\ L^{GL(r)} = L^{T}$.
\end{cor}

\begin{proof} Clearly, the left side  is contained in the right side.
The converse easily follows from Lemma \ref{Schur-Weyl} by induction on $r$.
\end{proof}

\begin{cor} \label{nice-too} The restriction of $\eta$ to the principal block ${\mathcal B}_{\mathbf 1}$ agrees with Serganova's functor $\eta_0$. Hence the restriction of $\eta$ to the principal block is fully faithful and defines an equivalence of $k$-linear abelian categories. This equivalence restricts to an equivalence between the even subcategories. In particular, we obtain \[ \eta (L(\underbrace{0,\ldots,0}_r,  \mu \vert -\mu)) =  L(\mu) \ .\] 
\end{cor}

\begin{proof} To show $\eta_0(X)=\eta(X)$, notice $\eta(X) \subseteq \eta_0(X) $. Since both functors are exact, to show equality it suffices  to show equality on simple objects. 
Since $(X^T)^{\vee} \cong (X^{\vee})^T$ and $(X^{GL(r)})^{\vee} \cong (X^{\vee})^{GL(r)}$,
both Corollary \ref{invariants} and Lemma \ref{SW} therefore imply $\eta_0(X)=\eta(X)$, and therefore the claim. The equivalence restricts to an equivalence between the even subcategories of the principal blocks since $res$ and Serganova's block equivalence commute with the (enriched) $DS$ functor and even objects correspond therefore to even objects. The statement about the image of $\eta(L(\lambda))$ follows from Serganova's description \cite{Serganova-blocks} of the block equivalence. 
\end{proof}

\begin{example} For $\Pi= L(0,\ldots,0,-1,\ldots,-1 \ | \ 1,\ldots,1)$ we get \[\eta(\Pi) = Ber_{n|n}^{-1}.\]
\end{example}

\begin{definition} We will use the following notation for the characters of $H_{m|n}$ given by the images of  $Ber_{m|n}$, its dual and $\Pi = B_{core}^{-1}$: \[ \mathcal{L}_{m|n}^{-1} := \overline{Ber}_{m|n}, \ \ \mathcal{L}_{m|n}^{} := \overline{Ber}_{m|n}^{-1} \text{ and } \mathcal{L}_{core} := \overline{B}_{core}^{-1}.\]    
\end{definition}



\section{The monoidal categories ${\mathcal M}_{m|n}$}

Recall that ${\mathcal T}_{m\vert n}^{\leq 0}$ denotes the intersection of $\mathcal{T}_{m|n}^{ev}$ with the full abelian category
of all objects in ${\mathcal T}_{m\vert n}^{ev}$ whose irreducible Jordan-H\"older
constituents $L(\lambda)$ have highest weights $\lambda$ with $\lambda_1\leq 0$.
If ${\mathcal T}_{m\vert n}^{< 0}$ denotes the intersection of $\mathcal{T}_{m|n}^{ev}$ with the full abelian category
of all objects in ${\mathcal T}_{m\vert n}^{\leq 0}$ whose irreducible Jordan-H\"older
constituents $L(\lambda)$ have highest weights $\lambda$ with $\lambda_r < 0$,
then ${\mathcal T}_{m\vert n}^{\leq 0}$ is a $k$-linear symmetric monoidal category with tensor ideal
${\mathcal T}_{m\vert n}^{< 0}$. It is not rigid, but satisfies all properties of a tensor category in the sense of the conventions 
in Section \ref{Conventions} except property (iii).
Furthermore notice $\eta({\mathcal T}_{m\vert n}^{< 0})=0$.

\begin{lem}\label{nicer} $\eta: {\mathcal T}_{m\vert n}^{\leq 0} \to {\mathcal T}_{n\vert n}^{\leq 0}$  
is a faithful exact $k$-linear monoidal functor. 
\end{lem}

\begin{proof} For the functor $\eta$ to be a monoidal functor, we have to show $(X\otimes Y)^G=X^G \otimes Y^G$. The inclusion $X^G \otimes Y^G \subseteq (X\otimes Y)^G$ is obvious. For the other inclusion, it suffices to show this on irreducible objects in $\mathcal{T}_{m|n}^{\leq 0}$. Since $\eta({\mathcal T}_{m\vert n}^{< 0})=0$, take $X = L(0,\ldots,0,\mu \, |\, \mu')$, $Y = L(0,\ldots,0,\nu \, | \, \nu')$. The weights in $X\otimes Y$ are bounded above by \[ (0,\ldots,0, \mu + \nu \,  | \, \mu' + \nu').\] If $\alpha,\beta$ are weights of $X$ and $Y$ respectively, then $\alpha + \beta$ are weights in $X \otimes Y$. If they are in the space of $G$-invariants, $\alpha + \beta$ is necessarily of the form $(0,\ldots,0,\ldots \, | \, \ldots)$ (first $r$ entries are zero). But their sum $\alpha + \beta$ can only be of this form if $\alpha$ and $\beta$ satisfy the same condition and are therefore also $G$-invariant.

\smallskip
Obviously $\eta$ is an exact functor since it is the invariant functor under a reductive group. By Lemma \ref{nice-too} it is faithful on simple objects, hence faithful
on all objects. If it is faithful on objects, an exact functor between abelian categories is faithful on morphisms.
\end{proof}

\begin{definition} We denote by $\mathcal{B}_{\one}$ the principal block (of either $\mathcal{T}_{m|n}$ or $\mathcal{T}_{n|n}$). We denote by $\mathcal{B}_{\one}^{\leq 0}$ the full abelian subcategory generated by irreducible objects \[L(0,\ldots,0,\mu_1,\ldots,\mu_n \, |  - \mu_n,\ldots,-\mu_1)\] (the $m\! >\! n$-case) or the full abelian subcategory generated by irreducible objects $L(\mu_1,\ldots,\mu_n \, | \, - \mu_n,\ldots,-\mu_1)$ with $\mu_1 \leq 0$. Their images in the semisimplifications will be denoted \[ {\mathcal M}_{m\vert n} = \overline{\mathcal{B}_{\one}^{\leq 0} \cap \mathcal{T}_{m|n}^{ev}} \ \ \text{ respectively } \ \ {\mathcal M}_{n\vert n} = \overline{\mathcal{B}_{\one}^{\leq 0} \cap \mathcal{T}_{n|n}^{ev}}.\]
\end{definition}

Note that these images are (non-rigid) abelian tensor subcategories of $\overline{\mathcal{T}_{m|n}^{ev}}$ and $\overline{\mathcal{T}_{n|n}^{ev}}$ in the sense of section \ref{Conventions}.

\begin{thm} \label{negative-half-equivalence} The functor $\eta: \mathcal{T}_{m|n}^{ev} \to \mathcal{T}_{n|n}^{ev}$ gives rise to a functor \[ \overline{\eta}: 
 {\mathcal M}_{m\vert n}  \longrightarrow  {\mathcal M}_{n\vert n} 
\] that induces an equivalence of $k$-linear abelian symmetric monoidal categories. 
\end{thm}

The proof is similar to the one of \cite[Lemma 5.11]{HW-tannaka}. Note that we are applying constructions of \cite{Andre-Kahn} that were defined for rigid tensor categories to a non-rigid setting. However the construction of the splitting is done in the rigid category $\mathcal{T}_{m|n}^{ev}$, and we only restrict then to the non-rigid subcategories.

\begin{proof} We define the ideal $\mathcal{I}^0$ via \[ \mathcal{I}^0(X,Y) = \{ f:X \to Y \ | \ f \text{ factorizes over a negligible object.} \} \] 
Obviously $\mathcal{I}^0$ is a tensor ideal of $\mathcal{T}_{m|n}^{ev}$. As for any tensor ideal $\mathcal{I}^0 \subset \mathcal{N}$ 
the quotient $\mathcal{T}_{m|n}^{ev}/\mathcal{I}^0 =: \mathcal{A}_{m|n}^{ev}$ becomes a rigid tensor category and
$\pi_{m|n}: \mathcal{T}_{m|n}^{ev}\to \mathcal{T}_{m|n}^{ev}/\mathcal{I}^0 = \mathcal{A}_{m|n}^{ev}$ a tensor functor. Under this tensor functor
an indecomposable object $X$  in  $\mathcal{T}_{m|n}^{ev}$ maps to zero in the quotient $\mathcal{A}_{m|n}^{ev}$ if and only if $\sdim(X) = 0$. The category $\mathcal{A}_{m|n}^{ev}$ is pseudoabelian since we have idempotent lifting due to the finite dimensionality of the Hom spaces. By the definition of $\mathcal{A}_{m|n}^{ev}$ and $\mathcal{T}_{m|n}^{ev}$, the dimension of each object in $\mathcal{A}_{m|n}^{ev}$ is a natural number and, contrary 
to  $\mathcal{T}_{m|n}^{ev}$, it does not contain any nonzero object that maps to an element isomorphic to  zero under the quotient functor $\mathcal{A}_{m|n}^{ev} \to \mathcal{A}_{m|n}^{ev}/\mathcal{N}$. Therefore $\mathcal{A}_{m|n}^{ev}$ satisfies conditions d) and g) in \cite[Theorem 8.2.4]{Andre-Kahn}. By \cite[Theorem 8.2.4 (i),(ii)]{Andre-Kahn} this implies 
that $\mathcal{N}( \mathcal{A}_{m|n}^{ev})$ 
equals the radical $\mathcal{R}( \mathcal{A}_{m|n}^{ev})$ of $ \mathcal{A}_{m|n}^{ev}$; note that $\mathcal{N}( \mathcal{A}_{m|n}^{ev}) = \mathcal{N}( \mathcal{T}_{m|n}^{ev})/\mathcal{I}^0$ and that 
${\mathcal N}(A,A)$ is a nilpotent ideal in $End(A)$ for any $A$ in $\mathcal{A}_{m|n}^{ev}$ by assertion b) of \cite[Theorem 8.2.4 (i),(ii)]{Andre-Kahn}. This allows to apply \cite[Theorem 13.2.1]{Andre-Kahn} to construct a monoidal section $s_{m|n}: \mathcal{A}^{ev}_{m|n}/\mathcal{N}(\mathcal{A}_{m|n}^{ev}) \to \mathcal{A}^{ev}_{m|n}$ for the tensor functor $\pi_{m|n}: \mathcal{A}^{ev}_{m|n} \to \mathcal{A}^{ev}_{m|n}/\mathcal{N}(\mathcal{A}_{m|n}^{ev})$.

We obtain the following diagram

\[ \xymatrix{ \mathcal{T}_{m|n}^{ev} \ar[r]^{\eta} \ar[d] & \mathcal{T}_{n|n}^{ev} \ar[d] \\ \mathcal{A}_{m|n}^{ev} \ar[d]^{\pi_{m|n}} & \mathcal{A}_{n|n}^{ev} \ar[d]^{\pi_{n|n}} \\ \mathcal{T}_{m|n}^{ev}/\mathcal{N} \ar@/^/@{.>}[u]^{s_{m|n}} & \mathcal{T}_{n|n}^{ev}/\mathcal{N} \ar@/^/@{.>}[u]^{s_{n|n}}.}\]

Without change of notation, we restrict $\eta$ and consider now  the restriction $\eta: {\mathcal T}_{m\vert n}^{\leq 0} \to {\mathcal T}_{n\vert n}^{\leq 0}$. By Lemma \ref{nicer} and Lemma \ref{eta-negligible} it induces a $k$-linear tensor functor \[ \widetilde\eta: \mathcal{A}_{m|n}^{ev} \longrightarrow \mathcal{A}_{n|n}^{ev}.\]

The composite tensor functor \[\overline{\eta}:=\pi_{n|n}\circ \widetilde\eta \circ s_{m|n}\]
defines a $k$-linear tensor functor 
$\overline{\eta}:  {\mathcal T}_{m\vert n}^{\leq 0}/\mathcal{N} \to {\mathcal T}_{n\vert n}^{\leq 0}/\mathcal{N}$.

\[ \xymatrix{ \mathcal{T}_{m|n}^{\leq 0} \ar[r]^{\eta} \ar[d] & \mathcal{T}_{n|n}^{\leq 0} \ar[d] \\ \mathcal{A}_{m|n}^{\leq 0} \ar[d]^{\pi_{m|n}} \ar[r]^{\widetilde\eta} & \mathcal{A}_{n|n}^{\leq 0} \ar[d]^{\pi_{n|n}} \\ \mathcal{T}_{m|n}^{\leq 0}/\mathcal{N} \ar@/^/@{.>}[u]^{s_{m|n}} \ar@{.>}[r]^{\overline{\eta}} & \mathcal{T}_{n|n}^{\leq 0}/\mathcal{N} \ar@/^/@{.>}[u]^{s_{n|n}}.}\]

Since $\eta$ is additive and $\mathcal{T}_{n|n}^{\leq 0}/\mathcal{N}$ is semisimple, $\overline{\eta}$ is additive and hence exact. 
The functor $\overline{\eta}$ restricts to a symmetric monoidal  functor $ \overline{\eta}: 
 {\mathcal M}_{m\vert n} 
\longrightarrow 
 {\mathcal M}_{n\vert n}$. By Corollary \ref{nice-too} it defines an equivalence of $k$-linear
abelian symmetric monoidal categories. 
\end{proof}

\subsection{Passage to the $+$-category}

The same argument can now be repeated for the full tensor subcategory $\mathcal{T}_{m|n}^+ \subset \mathcal{T}_{m|n}^{ev}$.

\begin{lem} \label{smalleta} The functor $\eta: \mathcal{T}_{m|n}^{\leq 0} \to \mathcal{T}_{n|n}^{\leq 0}$ induces a symmetric monoidal functor \[ {\mathcal T}_{m\vert n}^{\leq 0}/{\mathcal T}_{m\vert n}^{< 0} \to {\mathcal T}_{n\vert n}^{\leq 0} \] 
which restricts to a symmetric monoidal functor between the $+$-subcategories \[ \eta^{\leq 0}: ({\mathcal T}_{m\vert n}^{\leq 0}\cap \mathcal{T}_{m|n}^+)/({\mathcal T}_{m\vert n}^{< 0} \cap \mathcal{T}_{m|n}^+) \to {\mathcal T}_{n\vert n}^{\leq 0} \cap \mathcal{T}_{n|n}^+.  \] 
\end{lem}

\begin{proof} We only have to show that $\eta$ preserves the $+$-subcategory. This follows from Lemma \ref{nice-too} and Lemma \ref{nicer} since the tensor functor $\eta$ sends an irreducible representation to an irreducible representation, every object in the $+$-category is a direct summand of an iterated tensor product of irreducible modules and $\eta$ sends negligible modules to negligible modules.
\end{proof}

Notice, $\mathcal{B}_{\one}^{\leq 0}$ in the principal block  $\mathcal{B}_{\one}$ of ${\mathcal T}_{m\vert n}$
has trivial intersection with ${\mathcal T}_{m\vert n}^{< 0}$, thus by abuse of notation will tacitly be viewed
as being contained in ${\mathcal T}_{m\vert n}^{\leq 0}/{\mathcal T}_{m\vert n}^{< 0}$.
In view of Lemma \ref{smalleta} therefore the functor $\eta^{\leq 0}$ can be applied to
objects in the $k$-linear abelian symmetric monoidal category $$ {\mathcal M}^+_{m\vert n} \ =\  \overline{\mathcal{B}_{\one}^{\leq 0} \cap \mathcal{T}_{m|n}^{+}}\ .$$ 
This being said, we can view $\eta^{\leq 0}$ as a $k$-linear symmetric monoidal exact functor,
and its semisimplification $\overline \eta$
as in the next Proposition \ref{negative-half-equivalence-+}.

\begin{prop} \label{negative-half-equivalence-+} The functor $\eta: \mathcal{T}_{m|n}^{ev} \to \mathcal{T}_{n|n}^{ev}$ gives rise to an induced functor \[ \overline{\eta}: 
{\mathcal M}^+_{m\vert n}  
\longrightarrow 
{\mathcal M}^+_{n\vert n} 
\, \] which is an equivalence of $k$-linear symmetric monoidal abelian categories. 
\end{prop}

\begin{proof} As in the  case of  the $\mathcal{T}_{m|n}^{ev}$-category, $\eta$ is an equivalence $\eta: {\mathcal B}_{\one}^{\leq 0} \cap \mathcal{T}_{m|n}^{+} \to {\mathcal B}_{\one}^{\leq 0} \cap \mathcal{T}_{n|n}^{+}$ between the two negative halfs of the principal blocks. Hence the functor $\overline{\eta}$ constructed in the proof of Theorem \ref{negative-half-equivalence} induces an equivalence between the symmetric monoidal $k$-linear abelian categories ${\mathcal M}^+_{m\vert n}= \overline{{\mathcal B}_{\one}^{\leq 0} \cap \mathcal{T}_{m|n}^{+}}$ and ${\mathcal M}^+_{n\vert n}= \overline{{\mathcal B}_{\one}^{\leq 0} \cap \mathcal{T}_{n|n}^{+}}$. 
\end{proof}

\begin{example} \label{Compare} The object $\Pi = B_{core}^{-1}$ is in ${\mathcal T}_{m\vert n}^{\leq 0}$ with
image \[ \overline{B}_{core}^{-1} = \mathcal{L}_{core} \cong \one \boxtimes \overline{Ber}_{n|n}^{-1} = \one \boxtimes \mathcal{L}_{n|n}\] in $\overline{\mathcal{T}}_{m|n}$. Obviously  $$ \overline\eta(\mathcal{L}_{core}) = \mathcal{L}_{n\vert n} \ .$$ 
\end{example}

\begin{thm} \label{thm:principal} The images  in the semisimplification
of the principal blocks   ${\mathcal P}_{m\vert n} = \overline{\mathcal{B}_{\one} \cap \mathcal{T}_{m|n}^{+}}$ 
respectively ${\mathcal P}_{n\vert  n} = \overline{\mathcal{B}_{\one} \cap \mathcal{T}_{n|n}^{+}}$ are equivalent as tensor categories. 
\end{thm}

\begin{proof} By Proposition \ref{negative-half-equivalence-+} the $k$-linear exact monoidal functor $$\overline \eta: {\mathcal M}_{m\vert n}^+ \to {\mathcal M}_{n\vert n}^+$$ induces a monoidal
equivalence between ${\mathcal M}_{m\vert n}^+$ and ${\mathcal M}_{n\vert n}^+$ such that $\overline \eta( \mathcal{L}_{core}) = \mathcal{L}_{n\vert n}$ holds; see Example \ref{Compare}.
For the restriction of the fibre functor $\omega_{m\vert n}$ to ${\mathcal M}_{m\vert n}^+$
we claim that furthermore on ${\mathcal M}_{m\vert n}^+$ the following holds: 
$$ \omega_{m\vert n} = \omega_{n\vert n} \circ \overline\eta \ .$$
For this we have to restrict the fibre functor $\omega_{m\vert n}$ defining $H_{m\vert n}$ to the subcategory $Rep(H_{pr})$, and after that we have to restrict further to the subcategory 
 ${\mathcal M}:={\mathcal M}_{m\vert n}^+$. This double restriction defines the fibre functor $\omega_{\mathcal M}$ for $\mathcal M$.
By Lemma \ref{lem:fibre}  this double restriction  coincides with  the composition $\omega_{n\vert n} \circ \overline \eta$, where  $\omega_{n\vert n}$ is the fibre functor for $\overline{\mathcal T}_{n\vert n}$. Indeed, this follows from Lemma \ref{lem:fibre}, part (3) and the fact that $\overline \eta$ is induced from $\eta$ and that $\omega_{n\vert n}$ is induced
 from the functor $DS^n\!:\! {\mathcal T}_{n\vert n} \!\to \! {\mathcal T}_{n\vert n}$, both by passage to the semisimplification.
This being said, it now follows that the triples $({\mathcal M},\omega_{\mathcal M},{\mathcal M})$ defined by 
$({\mathcal M}_{m\vert n}^+,\omega_{m\vert n},\mathcal{L}_{core})$, respectively by  $({\mathcal M}_{n\vert n}^+, \omega_{n\vert n}, \mathcal{L}_{n\vert n})$, are tensor equivalent under the functor $\overline \eta$.

\medskip
For $(\mathcal T,\omega_{\mathcal T})=({\mathcal P}_{m\vert n},\omega_{m\vert n})$ 
and $(\mathcal T,\omega_{\mathcal T})=({\mathcal P}_{n\vert n},\omega_{n\vert n})$  and the triples $({\mathcal M}_{m\vert n}^+,\omega_{m\vert n},\mathcal{L}_{core})$ respectively  $({\mathcal M}_{n\vert n}^+, \omega_{n\vert n}, \mathcal{L}_{n\vert n})$
all assumptions of Lemma 
 \ref{reconstruction} hold.
 This allows to reconstruct
  $(\mathcal T,\omega_{\mathcal T})=({\mathcal P}_{m\vert n},\omega_{m\vert n})$ respectively 
 $(\mathcal T,\omega_{\mathcal T}) =({\mathcal P}_{n\vert n},\omega_{n\vert n})$ as Tannaka categories
 from the corresponding data  $({\mathcal M},  \omega_{\mathcal M},{\mathcal L})$.  
Therefore, taking into account Proposition \ref{negative-half-equivalence-+}, the 
reconstruction Lemma \ref{reconstruction} completes the proof. 
\end{proof}

Together with the discussion in Section \ref{reductions}, Theorem \ref{thm:principal} now  implies

\begin{cor} There exists an equivalence  $ \mathcal{T}_{m|n}^+/\mathcal{N} \cong Rep(GL(r)) \otimes_k (\mathcal{T}_{n|n}^+/\mathcal{N})$ of tensor categories 
such that $\mathcal{L}_{m\vert n}^{-1}$ corresponds to $det \boxtimes \mathcal{L}_{n\vert n}^{-1}$. By passage to the pro-reductive Tannaka groups we obtain an isomorphism \[ H_{m|n} \cong GL(r) \times H_{n|n}\ .\]
\end{cor}

\subsection{Examples}

\begin{example} Given $L(\lambda) = L(\lambda_1,\ldots,\lambda_r, \mu \ | \ -\mu)$ and $L(\lambda') = L(\lambda_1',\ldots,\lambda_r',\mu' \ | \ -\mu')$ suppose that $[\mu]$ and $[\mu']$ are not equivalent in the sense of \cite[Section 2.3]{HW-tannaka}. For $r=m\! -\! n$ they give rise to two different Tannaka groups $H_{\mu}$ and $H_{\mu'}$ of the form $GL(r) \times H_{\mu}$ and $GL(r) \times H_{\mu'}$. As explained in Section \ref{reductions}, decomposing $L(\lambda) \otimes L(\lambda')$ reduces to decompose the $GL(r)$-part and the $H_{\mu} \times H_{\mu'}$-part separately. Since the groups $H_{\mu}$ and $H_{\mu'}$ are disjoint, the tensor product $\rho_{H}^{\mu} \otimes\rho_{H}^{\mu'}$ corresponds to the external tensor product of representations. This defines an indecomposable representation. The decomposition behaviour depends therefore only on the $GL(r)$-part, hence is given by the Littlewood-Richardson rule. The different $GL(r)$-representations correspond one-one to the maximal atypical blocks, hence the Littlewood-Richardson rule describes the decomposition of $L(\lambda) \otimes L(\lambda')$ into blocks. If $\mu$ and $\mu'$ are equivalent, the groups coincide. In this case the $H_{\mu} \cong H_{\mu'}$-part decomposes into two or three summands (by \cite{HW-tannaka} it corresponds to a tensor product of (character twists of) standard representations or their duals of type $ABCD$).
\end{example}

\begin{example} For a list of examples in the $GL(3|3)$ and $GL(4|4)$-case we refer to \cite[Section 15]{HW-tannaka}.
\end{example} 

\begin{example} Consider $L(3,2,1,-2,-3,-4 \ | \ 4,3,2) \otimes L(4,2,1,0,-1,-2 \ | \ 2,1,0)$. These are already stable. The $GL(3)$ tensor product is of the form  
\[ \bigl(\rho_{cl}(4,2,1) \otimes \rho_{cl}(3,2,1)\bigr) \boxtimes \one \ ,\] where $\rho_{cl}(4,2,1) \otimes \rho_{cl}(3,2,1)$ decomposes into the 7 summands \[(\rho_{cl}(7, 4, 2) \oplus \rho_{cl}(7, 3, 3) \oplus \rho_{cl}(6, 5, 2) \oplus 2\rho_{cl}(6, 4, 3) \oplus \rho_{cl}(5, 5, 3) \oplus \rho_{cl}(5, 4, 4)) \boxtimes \one. \] This means that summands of non-vanishing superdimension appear exactly in the 6 different blocks of $GL(m|n)$ indexed by the above representations. The representations $[-2,-3,-4]$ and $[0,-1,-2]$ are of (SD)-type and equivalent by a $Ber_{3|3}^2$-twist. Hence their Tannaka groups coincide. By \cite[Section 11.6]{HW-tannaka} the associated Tannaka group is a $GSp(6)$ (since $\ell \neq 0$ and the pairing is symplectic rather than orthogonal). The tensor product \[ (\one_{cl} \boxtimes \overline{X}_{[-2,-3,-4]}) \otimes (\one_{cl} \boxtimes \overline{X}_{[0,-1,-2]})  = \one_{cl} \boxtimes ( \overline{X}_{[-2,-3,-4]} \otimes \overline{X}_{[0,-1,-2]})\] decomposes therefore in three indecomposable summands $I_1 \oplus I_2 \oplus I_3$ corresponding to the three irreducible summands in the decomposition of the natural representation of $GSp(6)$ (here both representations are character twists of it). All in all this leads to 21 indecomposable summands of non-vanishing superdimension. Of course we can now take one of the indecomposable summands (which is again labelled by a representation of $GL(3) \times GSp(6)$) and repeat this, so that we also get truncated fusion rules for all indecomposable modules of non-vanishing superdimension.
\end{example}

\subsection{Determinants} 

Even though we cannot exactly determine $H_{\mu}$ fully for some $L(\mu) \in \mathcal{T}_{n|n}^+$ of (SD) type, we can always compute the determinant.

As in \cite{HW-tannaka} we define \[\det(X_{\lambda}) = \Lambda^{\sdim(X_{\lambda})}(X_{\lambda}).\] Notice
$$    \det(X_\lambda)\ = \ I_\lambda \ \oplus \ \text{negligible} $$
is the sum of a unique indecomposable module $I_\lambda$ in $\mathcal{T}_{m|n}^+$ and a direct sum
of negligible indecomposable modules in $\mathcal{T}_{m|n}^+$. Clearly the superdimension of $det(X_{\lambda})$ in ${ \mathcal{T}_{m|n}^+}$ is one. Hence its image defines an invertible 
object of the representation category ${\overline{\mathcal T}}_{m|n} \cong Rep(H_{m|n})$ which we denote by $det(\overline{X}_{\lambda})$. 
In \cite[Section 13.1]{HW-tannaka} we defined the integer $\ell(\lambda)$ for any maximal atypical weight of $GL(n|n)$. If $X_{\lambda}$ is a stable module in the principal block of $GL(m|n)$, $\lambda = (0,\ldots,0,\mu \ | \ -\mu)$, we define $\ell(\lambda)$ to be the $\ell(\mu)$ for the corresponding representation of $GL(n|n)$ in the principal block.

\begin{thm} \label{det-formula} For stable $L(\lambda)$ in the principal block the module $\det(X_{\lambda})$ satisfies $$ \det(X_{\lambda}) \ = \ B_{core}^{\ell(\lambda)} \ \oplus \ \text{negligible}.$$ 
 
\end{thm}

\begin{proof} The theorem can be proven as in \cite[Theorem 14.3]{HW-tannaka} by induction on an order on the set of cup diagrams for a fixed block such that the representations with completely nested cup diagrams are the minimal elements. For stable modules these modules are groundstates in the sense of section \ref{sec:ground-states}. In the case of the principal block these are just the tensor powers $B_{core}^N$ and the claim is trivial. The specific power $\ell(\lambda)$ follows exactly as in \cite[Corollary 13.4]{HW-tannaka}. 
\end{proof}

This method allows in principle to compute $det(X_{\lambda})$ in general. Since $L(\lambda_N) \cong \Pi^N \otimes L(\lambda_0)$ by Lemma \ref{moving-lemma}, it suffices to compute the determinants of ground state representations. This is possible with the methods of the proof of Lemma \ref{lemma-ground}. However, a better way is to reduce the computation to the case $m\,=\, n$: For given $X_{\lambda} \in \mathcal{T}_{m|n}^+$, consider its image $L(\chi_{\lambda}) \boxtimes \overline{X}_{\mu} \in \Rep(GL(r) \times H_{n|n})$. Since the semisimplification functor is symmetric monoidal, it commutes with exterior powers, and we can therefore compute the determinant in the semisimplification. Using the determinant formula \[ \det(X \otimes Y) \cong det(X)^{\sdim(Y)} \otimes \det(Y)^{\sdim(X)} \] (see \cite[Proposition 1.11]{Deligne-tensorielles}), 
we obtain \[ \det(\overline{X}_{\lambda}) = \det(L(\chi_{\lambda}))^{\dim(V_{\mu})} \boxtimes \det(\overline{X}_{\mu})^{\dim(L(\chi_{\lambda})}.\]  Since up to a Berezin shift the second term $\det(\overline{X}_{\mu})$ is  computed in \cite[Theorem 14.3]{HW-tannaka} as (the image of) an explicit Berezin power, this yields a viable formula for the determinant also in the case $m\neq n$.

\section{Appendix: A variant of the $\eta$ functor}  
  
Analogues of the functor $\eta_0$ can be defined for other embeddings and invariant functors. Consider the embedding $\mathbb{G}_m \times GL(n\!-\! 1|n) \to GL(n|n)$ where $\mathbb{G}_m$ is embedded as the upper left $(1\times 1)$ block matrix. Again, let $$\eta: \T_{n|n} \to T_{n-1|n}\quad , \quad \eta = ()^{\mathbb{G}_m} \circ res$$ denote the restriction functor with respect to this embedding. As before let $\T_{m|n}^{\leq 0}$ denote the full abelian subcategory of $\mathcal{T}_{m|n}$ generated by all irreducible objects $L(\lambda)$ satisfying $\lambda_1 \leq 0$, and $\T_{m|n}^{<0}$ the full abelian subcategory with irreducible objects satisfying $\lambda_1 < 0$.

\begin{lem} The functor $\eta: \T_{n|n} \to T_{n-1|n}$ is an exact $k$-linear functor. It restricts to an exact $k$-linear monoidal functor $\eta: \T_{n|n}^{\leq 0} \to T_{n-1|n}^{\leq 0}$, such that $\eta(\T_{n|n}^{<0}) = 0$ holds.
\end{lem}

\begin{proof} Both $res$ and $()^{\mathbb{G}_m}$ are exact and $k$-linear. Obviously, the $\leq 0$ condition is preserved by both functors, such that $\eta(\T_{n|n}^{<0}) = 0$ holds. Preservation of tensor products is proven 
as in Lemma \ref{nicer}.
\end{proof}


\subsection{Kac modules} For $\g = \mathfrak{gl}(m|n)$ we define the lower/upper maximal parabolic subalgebras $\mathfrak{p}_{\pm} = \g_{(0)} \oplus \g_{(\pm1)}$ for the usual $\Z$-grading $\g = \g_{(-1)} \oplus \g_{(0)} \oplus \g_{(1)}$ where $\g_{(0)} \cong \mathfrak{gl}(m) \oplus \mathfrak{gl}(n)$ denotes the diagonal block matrices and $\g_{(\pm1)}$ denotes the upper/lower block matrices in $\g$ \cite{Kac-Rep}. We consider a simple $\g_{(0)}$-module as a $\mathfrak{p}_{\pm}$-module in which $\g_{(1)}$ respectively $\g_{(-1)}$ acts trivially. We then define the Kac module $V(\lambda)$ as the parabolically induced module \[ V(\lambda)  = Ind_{\mathfrak{p}_+}^{\g} L(\lambda) = U(\mathfrak{g}) \otimes_{U(\mathfrak{p}_+)} L(\lambda)\] where $L(\lambda)$ is the simple $\g_{(0)}$-module with highest weight $\lambda$ that has been trivially extended to a $\mathfrak{p}_+$-module. This is a $\mathfrak{g}$-module by the left action of $\mathfrak{g}$ on $U(\mathfrak{g})$. By the PBW theorem the Kac module is (as a vector space or $\mathfrak{p}_+$-module) given by \[ V(\lambda) = \Lambda^\bullet(\g_{(-1)}) \otimes L(\lambda).\] An element $X \in \mathfrak{g}$ acts on this as \begin{align*} & X \cdot (x_1 \wedge \ldots \wedge x_r \otimes v) \\ & = [X,x_1]x_2 \wedge \ldots \wedge x_r \otimes v - x_1 \wedge [X,x_2]x_3 \wedge \ldots \wedge x_r \otimes v \ldots \\ & + (-1)^{r+1} x_1 \wedge \ldots \wedge [X,x_r]\otimes v.\end{align*}

\begin{lem} \label{lem:kac} For Kac modules we have $\eta(V(0,\mu \, | \!-\mu,0)) \cong V(\mu \, | \! - \mu,0)$.
\end{lem}

\begin{proof} 
As a $k$-vectorspace $V(\lambda)$ is 
 a tensor product $\Lambda^\bullet(\g_{(-1)}) \otimes (L(0,\mu)\otimes L(-\mu,0)$. As a vectorspace $\g_{(-1)} \cong Mat_{n,n}(k)$.     
The action of $\mathfrak{gl}(n)\oplus \mathfrak{gl}(n)$ and the action of $\g_{(-1)}$
on this vector space are the obvious ones, and the $\frak p_+$-action was described above.

\medskip
By weight reasons, the subspace of $\mathbb G_m$-invariant vectors is given by the tensor product
$$ \Bigl( \Lambda^\bullet(\g_{(-1)}) \otimes (L(0,\mu)\otimes L(-\mu,0)\Bigr)^{\mathbb G_m} =
\Lambda^\bullet(\g_{(-1)})^{\mathbb G_m} \otimes (L(0,\mu)^{\mathbb G_m} \otimes L(-\mu,0) \ ,$$ 
using an argument similar to the one used in the proof of  \ref{nicer}. By Lemma \ref{Schur-Weyl} (2), then $$L(0,\mu)^{\mathbb G_m} =
L(\mu)\ .$$ 

\medskip
Furthermore, notice
$\Lambda^\bullet(\g_{(-1)})= \Lambda^\bullet(Mat_{n,n}(k))
= \Lambda^\bullet(k^n) \otimes   \Lambda^\bullet(Mat_{n-1,n}(k))$.
As a module under the even subgroup $\mathbb G_m \times GL(n\!-\! 1) \times GL(n)$, the first
factor $\Lambda^\bullet(k^n)$ is a trivial  $GL(n\!-\! 1)$-module and equipped with the natural action  
of the group $\mathbb G_m\times GL(n)$. The torus $\mathbb G_m$ acts trivially on
$\Lambda^\bullet(Mat_{n-1,n}(k))$.
Since also  $\Lambda^\bullet(Mat_{n-1,n}(k))$ is a trivial $\mathbb G_m$-module, this implies 
$$ \Lambda^\bullet(\g_{(1)})^{\mathbb G_m} = 
\Lambda^\bullet(k^n)^{\mathbb G_m} \otimes \Lambda^\bullet(Mat_{n-1,n}(k)) = \Lambda^\bullet(Mat_{n-1,n}(k)) \ .$$
Hence as a vector space, 
$$ \eta(V(0,\mu \, | \!-\mu,0)) = \Lambda^\bullet(Mat_{n-1,n}(k)) \otimes  L(\mu) \otimes L(-\mu,0) \ .$$
Here $Mat_{n-1,n}(k)$ represents the Lie superalgebra of lower block diagonal matrices in $\mathfrak{gl}(n\!-\! 1\vert n)$.
It now only remains to check, that the action of the Lie superalgebra of $GL(n\!-\! 1,n)$
on this vector space gives precisely the induced representation of the Kac  module  $V(\mu \, | \! - \mu,0)$
of $GL(n\!-\! 1\vert n)$. This is obvious for the action of $GL(n\!-\! 1)\times GL(n)$ and
the lower block diagonal matrices in $\mathfrak{gl}(n\!-\! 1,n)$, and the action of the upper block diagonal matrices carries over, hence our claim follows easily.

\end{proof}

\subsection{Composition factors of Kac modules}

By \cite{BS1} we have the following description of Kac-modules for $\mathcal{T}_{m|n}$: In the Grothendieck group 
the following holds \[ V(\lambda) = \sum_{\nu \subset \lambda} L(\nu),\]
where the notation $\nu \subset \lambda$ means the following: For a weight $\nu$ of $\T_{m|n}$ let $\underline{\nu}$ denote the resulting cup diagram where we have deleted all labels $\vee$, $\wedge$, $\circ$, $\times$ (so we only keep the cups at their specified positions). We can now superimpose this unlabelled cup diagram with the weight diagram of another weight, say $\lambda$, to form $\underline{\nu}\lambda$, i.e. we label the vertices on the numberline of the cup diagram $\underline{\nu}$ by $\lambda$. The notation $\nu \subset \lambda$ means that the resulting (labelled) cup diagram $\underline{\nu}\lambda$ is oriented in the sense of \cite{BS1}. For us this  simply implies that 
\begin{itemize}
\item Both $\lambda$ and $\nu$ have the same atypicality, and each cup in $\underline{\nu}$ has exactly one vertex labelled by $\vee$ (not necesarily the left vertex).
\item Both $\lambda$ and $\nu$ have the same core symbols $\circ$ and $\times$.
\end{itemize}

For example, if we take $\underline{\lambda}\lambda$, then we obtain just our usual (labelled) cup diagram.

The condition $\nu \subset \lambda$ implies $\nu \leq \lambda$ in the Bruhat order. In particular, all $\vee$'s in the weight diagram of $\nu$ have to be at vertices  $\leq$ than the $\vee$'s in the weight diagram of $\lambda$.

\begin{lem} \label{lem:kac-composition} Let $\lambda = (\lambda_1,\mu \,|  -\mu,-\lambda_1)$ denote a maximal atypical weight for $GL(n|n)$. Then:
\begin{enumerate}
\item The weight $\tilde{\lambda} = (\mu \, | \! -\mu,-\lambda_1)$ for $GL(n-1|n)$ has atypicality $n-1$. Its weight diagram is obtained from the one of $\lambda$ by moving the left $n-1$ symbols $\vee$ (i.e. except the rightmost) one position to the right and replacing the rightmost $\vee$ at position $\lambda_1$ by a $\circ$ at position $\lambda_1 + 1$.
\item The composition factors $L(\tilde{\mu})$ of the Kac module $V(\tilde{\lambda}) = V(\mu \, | \! -\mu,-\lambda_1)$ correspond 1:1 to the composition factors $L(\mu)$ of the Kac module $V(\lambda)$ satisfying $\mu_1 = \lambda_1$ via $L(\tilde{\mu}) \mapsto L(\lambda_1,\tilde{\mu})$.
\end{enumerate}
\end{lem}

\begin{proof}

The first assertion is easy to check using the definition of weight diagrams.
For the second assertion, consider now the cup diagrams of $\lambda$ and $\tilde{\lambda}$. Since the leftmost $n-1$ $\vee$ are up to a shift by 1 at the same vertices in the two weight diagrams, we write this symbolically as \[ \lambda = X  \, \vee, \quad , \quad \tilde{\lambda} = X \, \circ\] where $X$ stands for some specific configuration of $\vee$'s (identical in both diagrams). Now given any configuration of cups, for a cup diagram $\underline{\tilde{\nu}}$ coming from a weight $\tilde{\nu}$,  satisfying $\tilde{\nu} \subset \tilde{\lambda}$, the very same configuration of cups will give an oriented cup diagram for the first $n-1$ $\vee$'s coming from $\lambda$. The additional $\vee$ at position $\lambda_1$ in the weight diagram of $\lambda$ will only result in a cup between $(\lambda_1,\lambda_1 + 1)$ and not alter the cup configuration. The resulting weight, starting with $\tilde{\nu}  = (\mu \, | \! -\mu,\lambda_1)$, then is $\nu = (\lambda_1,\mu \, | \! -\mu,\lambda_1)$. It satisfies $\nu \subset \lambda$ and therefore yields a composition factor $L(\nu)$ of $V(\lambda)$. This gives an injection of all composition factors $L(\tilde{\nu})$ of $V(\tilde{\lambda})$ to the subset of composition factors $L(\nu)$ of $V(\lambda)$ satisfying $\nu_1 = \lambda_1$. 

\medskip
Conversely, given such a composition factor $L(\nu)$ with $\nu_1 = \lambda_1$, i.e. with a cup at the vertices $(\lambda_1,\lambda_1 + 1)$, all possible cup configurations that will result in a possible oriented diagram $\underline{\nu}\lambda$, are already exhausted by the remaining configurations of the other $n\!-\!1$ symbols $\vee$'s, which are in turn exactly the same ones as in $L(\tilde{\nu}) \vdash V(\tilde{\lambda})$. Hence every composition factor with $\nu_1 = \lambda_1$ is in the image of the previous injection.
\end{proof}

\begin{example} Consider the Kac module $V(\lambda_1,\lambda_2\, | \! - \lambda_2,-\lambda_1) \in \mathcal{T}_{2|2}$ with $\lambda_1 > \lambda_2$. It has four composition factors $ [\lambda_1,\lambda_2], \ [\lambda_1,\lambda_2 \!-\! 1], \ [\lambda_1\! -\! 1,\lambda_2], \ [\lambda_1\! -\! 1,\lambda_2- 1]$. The Kac module $V(\lambda_2\, | \! -\!\lambda_2,-\lambda_1)$ in $\T_{1|2}$ has the composition factors $L(\lambda_2 \,| \! -\!\lambda_2,-\lambda_1)$ and $L(\lambda_2 \!-\! 1 \, | \! - \lambda_2 \!+\! 1, -\lambda_1)$. They are the images of the constituents $[\lambda_1,\lambda_2], \ [\lambda_1,\lambda_2\! -\! 1]$ of the Kac module in $\T_{2|2}$. Notice that $\eta$ does not preserve  superdimensions.
\end{example}

\subsection{Computation of $\eta$}

Since $\eta(V(0,\mu \, | \! - \mu,0) = V(\mu \, | \! - \mu,0)$ holds by Lemma \ref{lem:kac-composition}, the exactness of $\eta$ allows us to compute $\eta(L(0,\mu \, | \! - \mu,0)$. We obtain

\begin{cor}  For irreducible 
objects $L(\lambda_1,\mu \, | \! - \mu,0)) $  in $\T_{n|n}^{\leq 0}$ we have $\eta(L(\lambda)) = 0$ for the functor $\eta: \T_{n|n}^{\leq 0} \to \T_{n-1|n}^{\leq 0}$ if $\lambda_1 < 0$. On the other hand, for $\lambda_1=0$ we obtain  \[ \eta(L(0,\mu \, | \! - \mu,0)) = L(\mu \, | \! - \mu,0) \ .\] 
\end{cor}

\begin{proof} The composition factors $L(\mu)$ of $V(0,\mu \,| \! - \mu,0)$ such that $\mu_1 \neq 0$ satisfy $\mu_1 <0$, hence are annihiliated by $\eta$. Furthermore $\eta(L(0,\mu \, | \! - \mu,0) \neq 0$. Hence the bijection of Lemma \ref{lem:kac-composition}, along with Lemma \ref{lem:kac}, implies the claim because $\eta$  is exact.
\end{proof}

\begin{remark} Note that $\eta$ is an exact tensor functor, but doesn't preserve superdimensions as it is not the restriction of a tensor functor between the tannakian categories $\mathcal{T}_{n|n}$ and $\mathcal{T}_{n-1|n}$. The simplest classical analogue is the functor $(\cdot )^{\mathbb{G}_m}$ on the category of polynomial representations of $GL(r)$ wich was studied in Section \ref{sec:sw-duality} which sends $L(\lambda_1,\ldots,\lambda_r)$ to $L(\lambda_1,\ldots,\lambda_{r-1})$ if $\lambda_r =0$ or to zero otherwise.
\end{remark}

\begin{remark} Other restriction/invariant functors can be studied in the same way. For example, passing from $\mathcal{T}_{n|n}$ to $\mathcal{T}_{n-k|n}$ can be simply obtained by iterating the $k=1$ construction.
\end{remark}


\end{document}